\documentclass[12pt]{article}
\usepackage{graphicx}
 \usepackage{mathptmx}      % use Times fonts if available on your TeX system
\usepackage{amsmath}
\usepackage{amssymb}
\usepackage{amsfonts}

\usepackage[usenames,dvipsnames]{color}
\usepackage{amsthm}
\newtheorem{prop}{Proposition}[section]{\bfseries}{\itshape}
\newtheorem{theo}[prop]{Theorem}{\bfseries}{\itshape}
\newtheorem{coro}[prop]{Corollary}{\bfseries}{\itshape}
\newtheorem{lemm}[prop]{Lemma}{\bfseries}{\itshape}

\newtheorem{remk}[prop]{Remark}{\bfseries}{\itshape}

\newcommand{\eqco}{\setcounter{equation}{0}}
\newcommand{\allco}{\eqco}
\newcommand{\la}{\lambda}
\newcommand{\Po}{{\cal P}}

\newcommand{\Q}{{\end{document}}}

\newcommand{\Bin}{{\rm Bin}}

\newcommand{\N}{\mathbb{N}}
\newcommand{\Z}{\mathbb{Z}}

\newcommand{\R}{\mathbb{R}}
\newcommand{\E}{\mathbb{E}}
\newcommand{\X}{\mathcal{X}}
\renewcommand{\Pr}{\mathbb{P}}
\renewcommand{\emptyset}{\varnothing}

\newcommand{\cL}{\mathcal L}

\newcommand{\essinf}{{\rm ess~inf}}
\renewcommand{\E}{\mathbb E \,}

\newcommand{\nalpha}{u}

\newcommand{\hH}{\hat{H}}
\newcommand{\cK}{{\cal K}}

\newcommand{\ph}{\varphi}
\newcommand{\LL}{{\cal L}}

\newcommand{\Vol}{{\rm Vol}}
\newcommand{\diam}{{\rm diam}}

\newcommand{\cF}{{\cal F}}

\newcommand{\eps}{\varepsilon}
\newcommand{\edm}{\end{displaymath}}
\newcommand{\be}{\begin{equation}}
\newcommand{\ee}{\end{equation}}
\newcommand{\bea}{\begin{eqnarray}}
\newcommand{\eea}{\end{eqnarray}}
\newcommand{\bean}{\begin{eqnarray*}}
\newcommand{\eean}{\end{eqnarray*}}

\renewcommand{\epsilon}{\varepsilon}
\newcommand{\bbS}{\mathbb{S}}
\newcommand{\dist}{\,{\rm dist}}
\begin{document}

\title{Largest nearest-neighbour link and connectivity threshold  in a polytopal random sample
\thanks{ Supported by EPSRC grant EP/T028653/1 }
%This research was funded, in part, by
%EPSRC Grant EP/T028653/1. A CC BY or equivalent licence
%is applied to the AAM arising from this submission,
%in accordance with the grant’s open access conditions.}
%%about the article that should go on the front page should be
}


\author{ Mathew D. Penrose \thanks{
	%First Author         \and
        %Second Author %etc.
%\institute{Mathew D. Penrose \at
  Department of
Mathematical Sciences, University of Bath, Bath BA2 7AY, United
Kingdom. 
              %first address \\
             \texttt{m.d.penrose@bath.ac.uk}             %  \\
}
	     \and Xiaochuan  Yang \thanks{Department of Mathematics, Brunel University London, Uxbridge, UB83PH, United Kingdom. 
	     \texttt{xiaochuan.yang@brunel.ac.uk  
	    ORCID:0000-0003-2435-4615}}
}
%            % \emph{Present address:} of F. Author  %  if needed
           %\and
           %S. Author \at
              %second address
%}

%\date{Revised  September-December 2021}

\maketitle

\begin{abstract}
	Let $X_1,X_2, \ldots $ be independent identically distributed  random points in a convex
	polytopal  domain $A \subset \mathbb{R}^d$.  Define the 
	{\em largest nearest neighbour link} $L_n$ to be the smallest
	$r$ such that every point of 
	$\X_n:=\{X_1,\ldots,X_n\}$ has another such point within distance $r$.  We obtain a strong law of large numbers for $L_n$ in the large-$n$ limit.   A related threshold, the \emph{connectivity threshold} $M_n$,  is the smallest $r$ such that the random geometric graph $G(\X_n, r)$ is connected. We show that
	as $n \to \infty$, 
	almost surely 
	$nL_n^d/\log n$ tends to a limit that depends on the geometry of
	$A$, and
	$nM_n^d/\log n$ tends to the same limit.
%Insert your abstract here. Include keywords, PACS and mathematical
%subject classification numbers as needed.
%\keywords{coverage threshold \and weak limit \and strong
% law of large numbers \and Boolean model \and Poisson point process}
% \subclass{
%	 %MSC code1 \and MSC code2 \and more}
%	 %MSC: 
%	 %AMS classifications:
%	 60D05 
%% (Geometric prob and stoch geom)
%	 \and
% 60G55
%	%(Point processes) 
%	\and
% 60F05 %(Central limit and other weak theorems) 
%	\and
% 60F15} %(strong limit theorems)
% %53A05 (surfaces in Euclidean and related spaces)
% %60G57 (random measures - maybe not)
\end{abstract}

%\section{Introduction}
%\label{intro}
%Your text comes here. Separate text sections with
%\section{Section title}
%\label{sec:1}
%Text with citations \cite{RefB} and \cite{RefJ}.
%\subsection{Subsection title}
%\label{sec:2}
%as required. Don't forget to give each section
%and subsection a unique label (see Sect.~\ref{sec:1}).
%\paragraph{Paragraph headings} Use paragraph headings as needed.
%\begin{equation}
%\end{equation}

%% For one-column wide figures use
%\begin{figure}
%% Use the relevant command to insert your figure file.
%% For example, with the graphicx package use
%  \includegraphics{example.eps}
%\end{figure}
%
%% For two-column wide figures use
%\begin{figure*}
%% Use the relevant command to insert your figure file.
%% For example, with the graphicx package use
%  \includegraphics[width=0.75\textwidth]{example.eps}
%\end{figure*}
%\begin{table}
%\begin{tabular}{lll}
%\hline\noalign{\smallskip}
%\end{tabular}
%\end{table}

\section{Introduction}
\label{SecIntro}

This paper is primarily concerned with the 
{\em connectivity threshold}
and 
{\em largest nearest-neighbour link} 
for a random sample $\X_n$ 
of $n$ points 
specified compact region $A$ in a $d$-dimensional Euclidean space. 
%We shall express our results in terms of the

The connectivity threshold, here denoted $M_n$, is defined to be the smallest $r$ such that  the random geometric graph $G(\X_n,r)$
is connected. For any finite $\X \subset \R^d $ the graph $G(\X,r)$ is defined
to have vertex set $\X$ with edges between those pairs of vertices 
$x,y$ such that 
%and for distinct $x,y \in \X$, $\{x,y\}$
%is an edge of $G(\X,r)$.
%, a graph with vertex set $\X_n$ and with edges $\{X_i,X_j\}$ whenever 
%$\dist(X_i,X_j)\le r$.
$\|x - y\|\le r$, where $\|\cdot\|$ is the Euclidean norm.
%is connected.
More generally, for $k\in\N$, the $k$-connectivity threshold $M_{n,k}$ is the smallest $r$ such that $G(\X_n,r)$ is $k$-connected (see the  definition in Section \ref{secdefs}).

 The largest nearest neighbour link, here denoted  $L_n$, is defined 
to be  the the smallest $r$ such that every
vertex in $G(\X_n,r)$ has degree at least 1.
%point of the sample is
 %within distance at most $r$ from the rest of the sample.
More generally, for $k \in \N$ with $k < n$,
 the {\em largest $k$-nearest neighbour link} $L_{n,k}$
is the smallest $r$ such that every 
vertex in $G(\X_n,r)$ has degree at least $k$.
These thresholds are random variables, because the locations of
the centres are random.
We investigate their probabilistic behaviour 
as $n$ becomes large.

We shall derive strong laws of large numbers
showing that
that $nL_{n,k}^d/\log n$
 converges  almost surely (as $n \to \infty$)
 to a finite positive limit,
and establishing the value of the limit.
Moreover we show 
that $nM_{n,k}^d/\log n$ converges to the same limit.
 These
strong laws carry over
to  more general cases where $k$ may vary with $n$, and  the distribution
of points may be non-uniform. We give results of this type
for $A$ a convex polytope.

Previous results  of this type (both for $L_{n,k}$ and for
$M_{n,k}$) were obtained for
$A$ having a smooth boundary, and for $A$ a $d$-dimensional
hypercube; see \cite{RGG}. It is perhaps not obvious from the
earlier results, however, how the limiting constant depends
on the geometry of $\partial A$, the
topological boundary of $A$,
for general polytopal $A$, 
which is quite subtle.

It  turns out,
	for example, that when $d=3$ and the points are
uniformly distributed over a polyhedron, the limiting behaviour of $L_n$
is determined by the angle of the sharpest edge if this angle is
less than $\pi/2$.
We believe
  (but do not formally prove here)  
 that if this angle exceeds $\pi/2$
 then the point of  $\X_n$ furthest
from the rest of  $\X_n$ is asymptotically uniformly distributed
over $\partial A$, but if this angle is less than $\pi/2$ the location
of this point in 
is asymptotically  uniformly distributed
over the union of those edges  which are sharpest.

Our motivation for this study is twofold. First, understanding the connectivity threshold in dimension two is vital in telecommunications, for example,
in  5G wireless network design, with the nodes of $\X_n$ representing
mobile transceivers (see for example \cite{BB}).
Second,  detecting connectivity is a fundamental step for detecting all other higher dimensional topological features in modern topological data analysis (TDA), where the dimension of the ambient space may be very high. See 
\cite{Bobrowski,BK18} for discussion of issues
related to the one considered here, in relation to TDA.
General motivation for considering random geometric graphs is 
discussed in \cite{RGG}.

While our main results are presented (in Section \ref{secdefs})
in the concrete setting of a polytopal sample in $\R^d$, our proofs
proceed via general lower and upper bounds (Propositions 
\ref{gammalem} and \ref{p:up}) that are  presented in the more general
setting of a random sample of points in a metric space satisfying certain
regularity conditions. This could be useful in possible future work dealing
with similar problems for random samples in, for example, a Riemannian manifold with boundary, a setting of importance in TDA.

%{\bf add intro for $M$} 

 %\section{Definitions and notation}
 \section{Statement of results}
 \label{secdefs}
\label{secLLN}

Throughout this paper, we work within the following mathematical framework.  Let $d \in \N$. Suppose we have the following ingredients:

\begin{itemize}
	\item A finite compact convex polytope
		$A \subset \R^d$ (i.e., one with finitely many faces).
	%	(Riemann measurability of a bounded set
%in $\R^d$ amounts to its boundary having zero Lebesgue measure). 
\item
	A Borel probability measure $\mu$ on $A$ with
		probability density function $f$.
	%We shall assume $\|f\|_\infty:= \essup_A f <\infty$ {\bf (can we avoid this?)}
	\item On a common  probability space
		$(\bbS,\cF,\Pr)$, a sequence  
$X_1,X_2,\ldots$ of  independent identically distributed
random $d$-vectors with common probability distribution $\mu$,
and also a unit rate Poisson counting
process   $(Z_t,t\geq 0)$, independent of $(X_1,X_2,\ldots)$
(so $Z_t$ is Poisson distributed with mean $t$ for each $t >0$).
\end{itemize}

For $n \in \N$, $t >0$,  let $\X_n:= \{X_1,\ldots,X_n\}$, and let
  $\Po_t:= \{X_1,\ldots,X_{Z_t}\}$.
These are the point processes that concern us here. Observe that $\Po_t$
is a Poisson point process in $\R^d $ with intensity measure $t \mu$
%where we set $\mu$ to be the distribution of $X_1$
 (see e.g. \cite{LP}).

For $x \in \R^d$ and $r>0$ set $B(x,r):= \{y \in \R^d:\|y-x\| \leq r\}$.
% where $\|\cdot\|$ denotes the Euclidean norm.
For $r>0$, let $A^{(r)}:=\{ x \in A: B(x,r) \subset A^o\}$, the
`$r$-interior' of $A$.

For any point set $\X \subset \R^d$ and any $D \subset \R^d$ we write
$\X(D)$ for the number of points of $\X$ in $D$,
 %and we use below the convention $\inf\{\} := +\infty$.
 and we use below the convention $\inf(\emptyset) := +\infty$.

Given $n, k \in \N$, and 
 $t \in (0,\infty)$,
 define the largest $k$-nearest neighbour link $L_{n,k}$
 by
\bea
L_{n,k} : =
 \inf ( \left\{ r >0: \X_n   (B(x,r)) \geq k+1 
~~~~ \forall x \in \X_n \right\}).
\label{Rnkdef}
\eea
Set $L_n : = L_{n,1}$.  Then
$L_n$ is the largest nearest-neighbour link.

We are chiefly interested in the asymptotic behaviour of $L_n$ 
 for large $n$.  More generally, we consider $L_{n,k}$ where $k$ may
vary with $n$.

%Observe that $\tilde{R}_{n,k} $ is the smallest $r$ such that $ B \cap A^{(r)} $
%is covered $k$ times by the balls of radius $r$ centred on the points
%of $\X_n$.  It can be seen that when $B=A$,
%the maximal $k$-spacing of the sample $\X_n$ (defined earlier) is 
%equal to $\theta_d \tL_n^d$, where 

Let
$\theta_d := \pi^{d/2}/\Gamma(1 + d/2)$,
the volume of the unit ball in $\R^d$.
Given $x,y \in \R^d$, we denote by $[x,y]$ the line segment from
$x$ to $y$, that  is, the convex hull of the set $\{x,y\}$.

%We now give  some further notation used throughout. For
%$D \subset \R^d$, let $\overline{D}$ and $D^o$
%denote the closure of $D$.  and interior of $D$, respectively. 
%Let $|D|$ denote the Lebesgue  measure (volume) of $D$, and
%$|\partial D|$ the perimeter of $D$, i.e. the
%$(d-1)$-dimensional Hausdorff measure of
%$\partial D$, when these are defined.  Write $\lglg t$ for $\log (\log t)$,
%	$t >1$.  Let $o$ denote the origin in $\R^d$.
%Set $\bH := \R^{d-1}\times [0,\infty)$ and 
% $\partial \bH := \R^{d-1}\times \{0\}$.
 
%Given two  sets $\X,\Y \subset \R^d$, we
% set  $ \X \triangle \Y := (\X \setminus \Y) 
%\cup (\Y \setminus \X)$, the symmetric difference between $\X$ and $\Y$.
%Also, we write $\X \oplus \Y$ for the set $\{x+y: x \in \X, y \in \Y\}$.
% Given also $x \in \R^d$ we write $x+\Y$ for $\{x\} \oplus \Y$.
%
%We write $a \wedge b$ (respectively $a \vee b$) for the minimum
%(resp. maximum) of any two numbers $a,b \in \R$.

Given  $m \in \N$ and functions
$f: \N \cap [m,\infty) \to \R$ and
$g: \N \cap [m,\infty) \to (0,\infty)$,
we write $f(n) = O(g(n))$ as $n \to \infty$,  
 if $\limsup_{n \to \infty }|f(n)|/g(n) < \infty$.
 %We write $f(n)= o(g(n))$ as $n \to \infty$
 %if $\lim_{n \to \infty} f(n)/g(n) =0$.
We write $f(n) = \Omega(g(n)) $ as $n \to \infty$
if $\liminf_{n \to \infty} (f(n)/g(n) ) >0$.
 %We write $f(n)= \Theta(g(n))$ as $n \to \infty$
 %if both $f(n)= O(g(n))$ and $f(n)= \Omega(f(n))$.
Given $s >0$ and  functions
$f: (0,s) \to \R$
and $g:(0,s) \to (0,\infty)$,
we write $f(r) = O(g(r))$ as $r \downarrow 0$
%or $g(r) = \Omega(f(r))$ as $ r \downarrow 0$,
if $\limsup_{r \downarrow 0} |f(r) |/g(r) < \infty$.
We write $f(r) = \Omega(g(r))$ as $r \downarrow 0$,
if $\liminf_{r \downarrow 0} (f(r) /g(r)) >0$.
% We write $f(r)= o(g(r))$ as $r \downarrow 0$
% if $\lim_{r \downarrow 0 } f(r)/g(r) =0$,
% and $f(r) \sim g(r)$ as $r \downarrow 0$
% if this limit is 1.

%\section{Strong laws of large numbers}
%\label{secLLN}
%\allco

%The results in this section  provide
%  strong laws of large numbers (SLLNs) for $L_{n}$. For these
%  results we relax the condition that $f$ be uniform on $A$.
%We give strong laws for $L_n$ when $A=B$ and $A $ is either
%smoothly bounded or a polytope.
%Also for general $A$ we
%give strong laws for $L_n$ when $\overline{B} \subset A^o$,
%and for $\tL_n$ 
%%when $ B=A$, 
%for general $B$.

%More generally, we consider $L_{n,k}$, allowing $k$ to vary with $n$.
	Throughout this section, assume we are given 
	a constant $\beta \in [0,\infty]$
	%{\bf (could consider $\beta = \infty $ too)}
	and a sequence $k: \N \to \N$ with
	%such that as $n \to \infty$ we have
	\bea
	%k(n)/\log n \to \beta;  ~~~~~ k(n)/n \to 0.
	\lim_{n \to \infty} \left( k(n)/\log n \right) = \beta;
	~~~~~ \lim_{n \to \infty} \left( k(n)/n \right) = 0.
	\label{kcond}
	\eea
We make use of the following notation throughout:
\bea
f_0 := \essinf_{x \in A} f(x); ~~~~~~~ f_1:= \inf_{x \in \partial A}
f(x);
\label{f0def}
\\
 H(t) := \begin{cases}  1-t + t \log t, ~~~ & {\rm if} ~ t >0 \\
	 1 ,  & {\rm if} ~ t =0.
 \end{cases}
	 \label{Hdef}
\eea
	Observe that $-H(\cdot)$ is unimodal with a maximum
value of $0$ at $t=1$. Given $a \in [0,\infty)$, we define
the
function $\hH_a: [0,\infty) \to [a,\infty)$ by
\bean
y = \hH_a(x)  \Longleftrightarrow y H(a/y) =x,~ y \geq a,
\eean
with $\hH_0(0):=0$.
Note that 
$\hH_a(x)$  is  increasing in  $x$,
and that 
$\hH_0(x)=x$ and $\hH_a(0)=a$.

	Throughout this paper, the phrase `almost surely' or 
	`a.s.' means `except on a set of $\Pr$-measure zero'.
For $n \in \N$, we use $[n] $ to denote $\{1,2,\ldots,n\}$.
 We write $f|_A$ for the restriction of $f$ to $A$.
	 %If $f_0=0,$ $ b>0$ we interpret $b/f_0$ as $+\infty$ in
 %the following limiting statements, and likewise for $f_1$.

Let $\Phi(A)$ denote the set of all faces of the polytope
$A$ (of all dimensions up to $d-1$).
Also, let $\Phi^*(A):= \Phi(A) \cup \{A\}$; it is sometimes
useful for us to think of $A$ itself as a face, of dimension $d$.
%For $\ph \in \Phi^*(A)$, we write $f|_\ph$ for the restriction of $f$ to $\ph$.

Given a  face $\varphi \in \Phi^*(A)$, denote the
dimension of this face by  $D(\varphi)$.
Then $0 \leq D(\varphi) \leq d$, and $\varphi$ is a 
$D(\varphi)$-dimensional polytope embedded in $\R^d$.
 Let $\varphi^o$ denote the relative interior of $\varphi$,
and set $\partial \varphi := \varphi \setminus \varphi^o$
(if $D(\ph)=0$ we take $\ph^o:= \ph$).
If $D(\ph) < d $ then set $f_\varphi := \inf_{x \in \varphi}f(x)$,
and if $\ph =A$ then set $f_\ph := f_0$.

Then there is a cone $\cK_\varphi$ in $\R^d$ such that every $x \in \varphi^o$
has a neighbourhood $U_x$ such that $A \cap U_x = (x+ \cK_\varphi) \cap U_x$.
Define the angular volume $\rho_\varphi$ of $\varphi$ 
to be the $d$-dimensional Lebesgue measure of
$\cK_\varphi \cap B(o,1)$. 

For example, if $\ph= A$ then $\rho_\ph = \theta_d$. If
$D(\varphi)=d-1$ then $\rho_\varphi = \theta_d/2$.
If $D(\varphi) = 0$ then $\varphi = \{v\}$ for some vertex $v \in  \partial A$,
and $\rho_\varphi$ equals 
the volume of $B(v,r) \cap A$, 
	divided by $r^d$,
for all sufficiently small $r$.
If $d=2$, $D(\varphi)=0$ and $\omega_\varphi$ denotes
the angle subtended by $A$ at the vertex $\varphi$, 
then $\rho_{\varphi} = \omega_\varphi/2$. If $d=3$ and
$D(\varphi)=1$, and  $\alpha_\varphi$ denotes
the angle subtended by $A$ at the edge $\varphi$ (which is the angle
between the two boundary  planes of $A$ meeting at $\varphi$), then
$\rho_{\varphi} = %(4\pi/3)\omega_\varphi/(2\pi) = 
2 \alpha_\varphi/3$.

\begin{theo}
	\label{thmpolytope}
	Suppose $A$ is a compact convex finite polytope in $\R^d$. 
	Assume that $f|_A$ is continuous at $x $ for all $x \in \partial A$,
	and that $f_0>0$.  Assume $k(\cdot)$ satisfies (\ref{kcond}).
	Then,
	almost surely,
	\begin{align}
\lim_{n \to \infty} nL_{n,k(n)}^d/ k( n) & =  
%	\max \left( \frac{1}{f_0 \theta_d}, 
	\max_{\varphi \in \Phi^*(A)} 
	\left(
	\frac{1}{f_\varphi \rho_\varphi }   
\right) 
%	\right) 
		&{\rm if} ~\beta = \infty;
	\label{0717a}
\\
\lim_{n \to \infty} nL_{n,k(n)}^d/ \log n & =  
	%\max \left( \frac{ \hH_\beta(1) }{f_0 \theta_d} ,
\max_{\varphi \in \Phi^*(A)} \left( \frac{ \hH_\beta( D(\varphi)/d)  }{ 
f_\varphi \rho_\varphi}  \right)
	%\right)
		 &{\rm if} ~\beta < \infty.
	%\nonumber \\
	\label{0717b}
\end{align}
\end{theo}

In the next three results,
we spell out some special cases of Theorem \ref{thmpolytope}.
\begin{coro}
\label{thpolygon}
Suppose that $d=2$,  $A$ is a convex polygon 
 and  $f|_A$ is continuous at $x$ for all $x \in 
	\partial A$. Let $V$ denote the set of vertices of $A$, and
	for $v \in V$ let $\omega_v$ denote the angle
	subtended by $A$ at vertex $v$.
	Assume (\ref{kcond}) holds
	with $\beta < \infty$.
Then, almost surely,
\bea
\lim_{n \to \infty} \left( \frac{ n  L_{n,k(n)}^2}{\log n}\right) =
	\max \left( \frac{\hH_\beta(1)}{\pi f_0}, 
	\frac{2 \hH_\beta(1/2)}{\pi f_1},
	\max_{v \in V}  \left( 
\frac{2  \beta }{  \omega_{v} f(v) } \right) 
	\right) .
	\label{polystrong2}
\eea
In particular, for any constant $k \in \N$,
$
\lim_{n \to \infty} \left( \frac{ n \pi L_{n,k}^2}{\log n}\right) =
 \frac{1}{f_0}. 
 $
\end{coro}

\begin{coro}
\label{thpoly}
	Suppose $d=3$ (so $\theta_d = 4\pi/3$),  $A$ is a convex polyhedron 
 and $f|_A$ is continuous at $x$ for all $x \in \partial A$. 
	Let $V$ denote the set of vertices of $A$, and
	 $E$  the set of edges  of $A$. For $e\in E$,
let $\alpha_e$ denote the angle subtended by $A$ at edge $e$,
and  $f_e$  the infimum of
$f $ over $e$.  
	For $v \in V$ let $\rho_v$ denote the
	angular volume of vertex $v$.
	Suppose (\ref{kcond}) holds 
with $\beta <  \infty$. Then, almost surely,
\bean
\lim_{n \to \infty} \left( \frac{ n  L_{n,k(n)}^3}{\log n}\right) =
	\max \left( \frac{\hH_\beta(1)}{\theta_3 f_0}, 
	\frac{2 \hH_\beta(2/3)}{\theta_3 f_1},
	\frac{3 \hH_\beta(1/3) }{ 2 \min_{e \in E} (\alpha_e f_e ) } 
	, 
\max_{v \in V}  \left( 
\frac{\beta }{ \rho_{v} f(v) } \right) 
 \right) .
\eean
In particular, if $\beta =0$ the above limit comes to
 $\max \left( \frac{3}{ 4 \pi f_0}, \frac{1}{\pi f_1},
\max_{e \in E}  \left( 
\frac{1 }{2 \alpha_e f_e } \right)
 \right) $.
\end{coro}
\begin{coro}[\cite{RGG}]
\label{thmbox}
Suppose $A = [0,1]^d$, and $f|_A$ is continuous at $x$ for all $x \in 
\partial A$. For $1 \leq j \leq d$ let $\partial_j$ denote
the union of all $(d-j)$-dimensional faces of $A$, 
	and let $f_j$ denote the infimum of
$f $ over $\partial_j$. Assume 
	(\ref{kcond})  with $\beta < \infty$. 
  Then
\bea
\lim_{n \to \infty} \left( \frac{ n  L_{n,k(n)}^d}{\log n}\right) =
	\max_{0 \leq j \leq d} \left( \frac{2^j \hH_\beta(1-j/d)  }{ \theta_d f_j } \right),
	~~~~~ a.s.
\label{0505b}
\eea
\end{coro}

It is perhaps worth spelling out what the preceding results mean in the
special case where $\beta =0$ (for example, if $k(n)$ is a constant)  and also
$\mu$ is the uniform distribution on $A$ (i.e. $f(x) \equiv f_0$  on $A$).
In this case, the right hand side of (\ref{0717b}) comes to 
%$f_0^{-1}%\max(1/ \theta_d,
$\max_{\varphi \in \Phi^*(A)} \frac{D(\varphi)}{(d f_0 \rho_\varphi )}$.
The limit in (\ref{polystrong2}) comes to $1/(\pi f_0)$,
while the limit in Corollary \ref{thpoly} comes to 
$f_0^{-1} \max[ 1/\pi, \max_e (1 /(2 \alpha_e))]$.

So far we have only presented results for the largest $k$-nearest neighbor link.
A closely related threshold is the \emph{$k$-connectivity threshold} defined by 
\begin{align*}
M_{n,k}:= \inf\{r>0: G(\X_n,r) \mbox{ is } k \mbox{-connected} \},
\end{align*}
where a graph $G$ of order $n$ is said to be $k$-connected ($k<n$) if $G$ cannot be disconnected by the removal of at most  $k-1$ vertices.  
Set $M_{n,1}=M_n$. Then $M_n$ is the connectivity threshold. 

Notice that  for all $k,n$ with $k<n$ we have
\begin{align}
L_{n,k}\le M_{n,k}.
\label{e:LleM}
\end{align}
 Indeed, if $r<L_{n,k}$, then there exists $i\in[n]$ such that $\deg X_i <k$ in $G(\X_n,r)$.  Then the removal of all vertices adjacent to $X_i$  disconnects
 $G(\X_n,r)$, implying that $r< M_{n,k}$. This proves the claim.

Our second main result shows that  $(M_{n,k}/L_{n,k}) \to 1 $
almost surely as $n \to \infty$. For this result we need $d \geq 2$.
%is  asymptotic to and $L_{n,k}$ almost surely.
  
  \begin{theo}
  \label{t:M}
	  Suppose $d \geq 2$.
  Suppose $A$ is a compact convex finite polytope in $\R^d$. 
	Assume that
	$f|_A$ is continuous at $x $ for all $x \in \partial A$,
	and that $f_0>0$.
	Assume $k(\cdot)$ satisfies (\ref{kcond})
	 % with $\beta < \infty$.
	Then,
	almost surely,
	  \begin{align}
		  \lim_{n \to \infty} n M_{n,k(n)}^d/ k( n)   & = 
		 % \max \left( \frac{1}{f_0 \theta_d}, 
		  \max_{\varphi \in \Phi^*(A)} \left(
	\frac{1}{f_\varphi \rho_\varphi }   
\right) 
		  %\right) 
%~~~~~ ~~
%~~~~~ ~~~~~
		  &{\rm if} ~\beta = \infty;
	\label{07172}
\\
		  \lim_{n \to \infty} n M_{n,k(n)}^d/ \log n &  =
		 % \max \left( \frac{ \hH_\beta(1) }{f_0 \theta_d} ,
	\max_{\varphi \in \Phi^*(A)} \left( \frac{
		\hH_\beta( D(\varphi)/d)  }{ 
f_\varphi \rho_\varphi}  
	\right) 
		  %\right)
		 % ~~~~~~ 
		  &{\rm if} ~\beta < \infty.
	%\nonumber \\
	\label{07171}
	  \end{align}
  \end{theo}

\begin{remk}{\rm
One can spell out consequences of Theorem \ref{t:M} in dimensions $d=2,3$ and the case of $[0,1]^d$	 with exactly the same statement as in Corollaries \ref{thpolygon}-\ref{thmbox}. }
\end{remk}

\begin{remk}{\rm
	Theorems \ref{thmpolytope} and \ref{t:M} extend  earlier work found
	in \cite{RGG} on the case where $A$ is the unit cube, to
	more general polytopal regions. The case where $A$ has a smooth
	boundary is also considered in \cite{RGG} (in this case with
also  $k(n)=$
	const., the result was first given
	in \cite{LNNL} for $L_{n,k}$ and in \cite{pen99} for $M_{n,k}$).}
\end{remk}

\begin{remk}
	{\rm In \cite{CovPaper}, similar results are given  
for	the {\em $k$-coverage threshold} $R_{n,k}$, which is given by 
\bea
R_{n,k} : = \inf \left\{ r >0: \X_n   (B(x,r)) \geq k 
~~~~ \forall x \in A  \right\};
~~~ n,k  \in \N.
\label{oldRnkdef}
\label{oldRdashdef}
\eea
	Our  results here, together with \cite[Theorem 4.2]{CovPaper},
	show that both $L_{n,k(n)}$ and $M_{n,k(n)}$ 
	are asymptotic to $R_{n,k(n)}$ almost surely, as $n \to \infty$.
	}
\end{remk}

\section{Proofs}

%\section{Strategy of proofs}
\label{secstrategy}
\label{secLLNpfs}
\allco
In this section we prove the results stated in Section \ref{secLLN}.
Throughout this section we are assuming we are  given a constant 
$\beta \in [0,\infty]$ and a sequence $(k(n))_{n \in \N}$ satisfying 
(\ref{kcond}).  Recall that $\mu$ denotes the distribution of $X_1$,
and this has a density $f$ with support $A$, and that 
$L_{n,k}$ is defined at (\ref{Rnkdef}).  Recall that $\hH_\beta(x)$ is
defined to be the $ y \geq \beta $ such that $y H(\beta/y) =x$, where
$H(\cdot)$ was defined at (\ref{Hdef}).

%We shall now prove the results
%%We first give an overview of the strategy for the proofs, in Section
%%\ref{secLLNpfs}, of the strong laws of large numbers
%that were stated in Section \ref{secLLN}.

For $n \in \N$ and $p \in [0,1]$ let $\Bin(n,p)$ denote
  a binomial random variable with parameters $n,p$.
  Recall that $H(\cdot)$ was defined at (\ref{Hdef}), 
	  and
  $Z_t$ is  a Poisson$(t)$ variable for $t>0$.
  The proofs in this section  rely heavily on the
  following lemma. %{\bf (might not need all of these?)}.
\begin{lemm}[Chernoff bounds]
\label{lemChern}
Suppose  $n \in \N$, $p \in (0,1)$, $t >0$ and $0 \leq k < n$. 
%Set $\eta := np$.

	(a) If $k \geq np$ then $\Pr[ \Bin(n,p) \geq k ] \leq \exp\left( 
	- np H(k/(np) )
\right)$. 

	(b) If $k \leq np$ then $\Pr[ \Bin(n,p) \leq k ] \leq \exp\left( - np H(k/(np))
\right)$. 

	(c) If $k \geq e^2 np$ then 
 $\Pr[ \Bin(n,p) \geq k ] \leq \exp\left( - (k/2)
	\log(k/(np))\right) \leq e^{-k}$.

	(d) If $k < t$ then $\Pr[Z_t \leq k ] \leq
	\exp(- t H(k/t))$.

	(e) If $k \in \N$ then $\Pr[Z_t = k ] \geq
	(2 \pi k)^{-1/2}e^{-1/(12k)}\exp(- t H(k/t))$.
\end{lemm}
\begin{proof}
	See e.g. \cite[Lemmas 1.1, 1.2 and 1.3]{RGG}.
	%for (a), (b) and (c), and \cite[Lemma1.3]{RGG} for (d).
\end{proof}

\subsection{{\bf A general lower  bound}}
\label{subsecgenbds}
%\subsection{General lower and upper bounds}

In this subsection 
we present an asymptotic lower bound on $L_{n,k(n)}$, not requiring any extra
assumptions on $A$. In fact, $A$ here can be any metric space
endowed with a Borel probability  measure $\mu$ which satisfies the following 
for some $\epsilon' >0 $ and some $d>0$:
\begin{align}
\label{e:vol_lowB}
\mu(B(x,r))\ge \epsilon' r^d, \quad \forall ~ r\in (0,1), x\in A.
\end{align} 
 The definition of $L_{n,k}$ at (\ref{Rnkdef})
carries over in an obvious way to this general setting.

Later, we shall derive the results
stated in Section \ref{secLLN}
by  applying the results of this subsection to the different 
regions within $A$ (namely interior, boundary, and
lower-dimensional faces).

Given $r >0, a>0$,   define the `packing number' $ \nu(r,a) $
 be the largest number $m$ such that there exists a collection of $m$ disjoint closed balls of radius $r$ centred on points of $A$,
each with $\mu$-measure at most $a$.
%The proof of the following lemma implements,
%for a general metric space,
%the strategy outlined in Section \ref{secstrategy}. 

\begin{prop}[General lower bound]
	\label{gammalem}
	Assume  \eqref{e:vol_lowB} with $d,\epsilon'>0$.
	Let $a >0, b \geq 0$. Suppose
	$\nu(r,a r^d) = \Omega (r^{-b})$ as $r \downarrow 0$.
	Assume \eqref{kcond}.
	Then
	almost surely, if $\beta = \infty$ then
	$\liminf_{n \to \infty} \left(n  L_{n,k(n)}^d/k(n) \right) \geq
	1/a$. If $\beta < \infty$ then
	$\liminf_{n \to \infty} \left(n  L_{n,k(n)}^d/\log n \right) 
	\geq a^{-1} \hH_\beta(b/d)$, almost surely.
\end{prop}
\begin{proof}
	%{\bf (Add $\beta = \infty $ argument if include this.)}
First suppose $\beta = \infty$.  Let $\nalpha \in (0,1/a )$.  Set 
$r_n := \left( \nalpha k(n)/n \right)^{1/d}$, $n \in \N$.  By (\ref{kcond}), 
$r_n \to 0$ as $n \to \infty$.  Then, given $n$ sufficiently large, we have 
$\nu(r_n,ar_n^d ) >0$ so  we can find $y_n \in A$ such that 
$ \mu(B(y_n,r_n)) \leq a r_n^d$, and hence
$ n \mu (B(y_n,r_n)) \leq a \nalpha k(n) $.  If 
$k(n) \leq e^2 n \mu(B(y_n,r_n))$ (and hence 
$n \mu(B(y_n,r_n)) \geq e^{-2} k(n)$), then since $\X_n(B(y_n,r_n))$ is binomial
with parameters $n$ and $\mu(B(y_n,r_n))$, by Lemma \ref{lemChern}(a)
we have that
	\bean
%	\Pr[ L_{n,k(n)} \leq r_n] \leq
	\Pr[ \X_n(B(y_n,r_n) ) \geq k(n) ] & \leq & \exp \left( - n \mu(B(y_n,r_n)) H \left(\frac{k(n)}{n \mu( B(y_n,r_n))} \right) \right)
	\\
	& \leq & \exp \left( - e^{-2} k(n) H \left( (a \nalpha)^{-1} \right) \right),
	\eean
while  if $k (n) > e^2 n \mu(B(y_n,r_n))$ then by Lemma \ref{lemChern}(c),
	%$\Pr[ L_{n,k(n)} \leq r_n] \leq e^{-k(n)}$.
	$\Pr[\X_n (B(y_n,r_n)) \geq k(n)] \leq e^{-k(n)}$.
Therefore $ \Pr[\X_n (B(y_n,r_n)) \geq k(n)] $ is summable in $n$ because 
	$k(n)/\log n \to \infty$ as $n \to \infty$ by (\ref{kcond}). 
	
Let $\delta_0 \in(0,1)$. By (\ref{e:vol_lowB}) $\mu(B(y_n,\delta_0 r_n)
\geq \eps' \delta_0^d u k(n)/n$.  Therefore by Lemma \ref{lemChern}(b), 
	\linebreak $\Pr[\X_n(B(y_n,\delta_0 r_n)) =0]
	\leq \exp (-\eps' \delta_0^d u k(n))$, which is summable in $n$.

	Thus by the Borel-Cantelli lemma, almost surely  event
	$F_n:= \{\X_n (B(y_n,r_n)) < k(n) \} \cap \{\X_n(B(y_n,\delta_0 r_n)) >0\}$ 
	occurs for all but finitely many $n$. But if $F_n$ occurs then
	$L_{n,k(n)} \geq (1-\delta_0) r_n$ so that $n L_{n,k(n)}^d/k(n)
	\geq (1-\delta_0)^d u $.
	%and hence $ \liminf n  L_{n,k(n)}^d /k(n)  \geq \nalpha$.
	This gives the result for $\beta = \infty$. 

	Now suppose instead that $\beta < \infty$.
	Suppose first that
	$b=0$, so that $\hH_\beta(b/d) = \beta$.
	Assume that   $\beta >0$ 
	(otherwise the result is trivial).
	Choose $\beta' \in (0, \beta)$.  Let $\delta > 0$
with $\beta' < \beta- 2\delta$ and with
$
\beta' H \left( \frac{ \beta - 2 \delta}{\beta' } \right) > \delta.
$
This is possible because $H(\beta/\beta')>0$ and $H(\cdot)$ is continuous.
For $n \in \N$, set $r_n := (  (\beta' \log n)/ (a n) )^{1/d}$. Also set
$k'(n)= \lceil (\beta - \delta) \log n \rceil$, and
$k''(n)= \lceil (\beta - 2 \delta) \log n \rceil$.
By assumption $\nu(r_n,a r_n^d) = \Omega(1)$, so for all $n$ large enough, 
	we can (and do) choose $x_n \in A$ such that
$
	n \mu(B(x_{n},r_n)) \leq n a r_n^d = \beta' \log n. 
$
Then by a simple coupling, and Lemma \ref{lemChern}(a),
\bean
\Pr[
\X_{n} ( B(x_{n},r_n) ) \geq k''(n)
	]
	& \leq & 
	\Pr \left[ \Bin\left(n,
	%\frac{\beta' \log n}{n}
	(\beta' \log n)/n)
	\right) \geq k''(n) \right]
\\
	& \leq & \exp \left( - \left( 
	\beta' \log n
	\right) 
H \left( \frac{\beta- 2 \delta }{\beta' } \right)  
\right) 
\leq n^{- \delta}.
\eean
	 Let $\delta' \in (0,1)$.
	By \eqref{e:vol_lowB}, for $n $ large enough and all $x \in A$, 
	\bean
	n \mu(B(x, \delta' r_n)) \geq  n \eps' (\delta' r_n)^d 
	=  \eps' (\delta')^d (\beta'/a) \log n
	\eean
	so that by Lemma \ref{lemChern}(b), 
	$\Pr [\X_n(B(x,\delta' r_n)) =0] \leq n^{-\eps'(\delta')^d\beta'/a}$. 
	%by Lemma \ref{LemCher}, $\Pr[

	Now choose $K \in \N$ such that $\delta K >1$ and
	$K\eps' (\delta')^d \beta'/a >1$. For $n \in \N$ set $z(n):=n^K$.
	For all large enough $n$ we have
	$k'(z(n)) \geq k''(z(n+1))$, so by
	the preceding estimates, 
	\bean
	\Pr[\X_{z(n+1)}(B(x_{z(n+1)},r_{z(n+1)})) \geq k'(z(n))]
	~~~~~~~~~~~~~~~~~~~~~
	~~~~~~~~~~~~~
	~~~~~~~~~~~~~
	\\
	~~~~~~~~~~~~~
	\leq
	\Pr[\X_{z(n+1)}(B(x_{z(n+1)},r_{z(n+1)})) \geq k''(z(n+1))]
	\leq (n+1)^{-\delta K},
	\eean
and since $x_{z(n+1)} \in A $, also
$\Pr[ \X_{z(n)}(B(x_{z(n+1)},\delta' r_{z(n)})) =0] 
	\leq n^{- \eps' (\delta')^d \beta' K/a}$.
Both of these upper bounds are summable in $n$, so by the Borel-Cantelli lemma,
almost surely   for all large enough $n$ we have the event
$$
\{\X_{z(n+1)}(B(x_{z(n+1)},r_{z(n+1)})) < k'(z(n))\} \cap 
\{\X_{z(n)}(B(x_{z(n+1)},\delta' r_{z(n)})) >0 \}.
$$
Suppose the above event occurs and suppose 
$m \in \N$ with $z(n ) \leq m \leq z(n+1)$. Note that 
$r_{z(n+1)}/r_{z(n)}\to 1$ as $n\to\infty$.  Then, provided $n$ is large enough,
$$
L_{m,k'(z(n))} \geq r_{z(n+1)}- \delta' r_{z(n)} \geq (1-\delta')^2 r_{m},
$$
and moreover $k'(z(n)) \leq k(m)$ so that $L_{m,k(m)} \geq (1- \delta')^2 r_m$. Hence it is almost surely the case that
$$
\liminf_{m \to \infty} (m L_{m,k(m)}^d/\log m) \geq (1- \delta')^{2d}
\liminf_{m \to \infty} (m r_{m}^d/ \log m) = (1-\delta')^{2d} a^{-1}\beta',
$$
	 and this yields the result for this case.
%
%Hence there is some  $\delta'' >0$ such that for large $n$,
%\bean
%\Pr[ n L_{n,k'(n)}^d \leq (1-\delta')^d(\beta'/a) \log n ]
%= \Pr[ L_{n,k'(n)} \leq (1-\delta') r_n] \\
%\leq \Pr[ \X_n(B(x_n,r_n)  \geq k_n+1 ]
%+ \Pr[ \X_n(B(x_n,\delta' r_n)  =0 ]
%\leq n^{-\delta''}.
%\eean
%This tends to zero, which gives us the result
%we want, but only in probability.

Now suppose instead that $\beta < \infty$ and $b > 0$.
Let $\nalpha \in( a^{-1}\beta,  a^{-1} \hH_\beta(b/d))$; note that this
implies $\nalpha a  H(\beta/(\nalpha a)) < b/d$.  Choose $\eps >0$ such
that $(1+ \eps) \nalpha a H(\beta/(\nalpha a)) < (b/d)-9 \eps$.
Also let $\delta' \in (0,1)$.

For each $n \in \N$ set $r_n= (\nalpha (\log n)/n)^{1/d}$.
Let $m_n := \nu(r_n,a r_n^d)$, and choose $x_{n,1},\ldots,$ $x_{n,m_n} \in A$
such that the balls $B(x_{n,1},r_n),\ldots,B(x_{n,m_n},r_n)$ are
pairwise disjoint and each have $\mu$-measure at most $ar_n^d$.

Set $\lambda(n):= n+ n^{3/4}$ and $\lambda^-(n) := n- n^{3/4}$.
For $1 \leq i \leq m_n$, if $k(n) \ge 1$ then by a simple coupling, 
and Lemma \ref{lemChern}(e), 
\bean
	\Pr[ \Po_{\lambda(n)}( B(x_{n,i},r_n) ) \leq k(n) ]
	\geq \Pr[ Z_{\lambda(n) ar_n^d} \leq k(n)]
	%\nonumber 
	\\
\geq 
 \left( \frac{e^{-1/(12k(n))}}{ \sqrt{2 \pi k(n)}} \right) 
%~~~~~~~~~~~~~~~~~~~~~~~~~~~~
%\nonumber \\
%~~~~~~~~~~~~~~~~~~~~~~~~~~~  \times  
\exp 
\left( - \lambda(n) a r_n^d  
	H \left( \frac{k(n)}{\lambda(n) a r_n^d } \right)  
\right).
\eean
Now $\lambda(n)r_n^d/\log n \to \nalpha$ so by (\ref{kcond}),
$k(n)/(\lambda(n) a r_n^d) \to \beta/(\nalpha a)$ as $n \to \infty$. Thus by
the continuity of $H(\cdot)$, provided $n$ is large enough,
	for $1 \leq i \leq m_{n}$,
\bean
\Pr[\Po_{\lambda(n)}(B(x_{n,i},r_n))  \leq  k(n)]  
~~~~~~~~~~~~~~~~~~~~~~~~~~~  
\nonumber \\
\geq 
 \left( \frac{e^{-1/12} } {\sqrt{2 \pi (\beta +1 ) \log n }} \right)
%%~~~~~~~~~~~~~~~~~~~~~~~~~~~  \times
  \exp 
\left( - (1+ \eps)  a \nalpha   
H \left( \frac{\beta}{a \nalpha   } \right) \log n  
\right) .
%\label{0319b5}
\eean
Hence, by our choice of $\eps$, there is a constant $c >0 $ such that for  all
large enough $n$ and all $i \in [m_n]$ we have
\bea
\Pr[\Po_{\lambda(n)}(B(x_{n,i},r_n))  \leq  k(n)]  \geq c (\log n)^{-1/2}
	n^{   9 \eps - b/d  } \geq n^{8\eps -b/d}. 
\label{0319b5}
\eea
Since $x_{n,i} \in A$, by \eqref{e:vol_lowB}, for $n $ large enough and
$1 \leq i \leq m_n$ we have
    $ \mu(B(x_{n,i}, \delta' r_n)) \geq   \eps' (\delta' r_n)^d $
    (as well as $ \mu(B(x_{n,i},  r_n)) \leq   a  r_n^d $).
	Thus, given the value of
$\Po_{\la(n)}(B(x_{n,i},r_n))$, 
the value
of $\Po_{\la^-(n)} (B(x_{n,i},\delta'r_n))$ 
is binomially distributed with probability parameter bounded away from zero.
Also $\max_{1\leq i \leq m_n}
\E[\Po_{\la(n)}(B(x_{n,i},r_n))]$ tends to infinity as $n \to \infty$. 
Therefore there exists $\eta >0$ such that for all large enough $n$,
defining the event
$$
E_{n,i} := \{\Po_{\la(n)} (B(x_{n,i},r_n)) \leq k(n) \}
\cap \{ \Po_{\la^-(n)}(B(x_{n,i},\delta' r_n) \geq 1 \},
$$
we have for all large enough $n$ that
$$
\inf_{1 \leq i \leq m_n} \Pr[E_{n,i}|\Po_{\la(n)}(B(x_{n,i},r_n)) \leq k(n)]
\geq \eta.
$$
	Hence, setting $E_n:= \cup_{i=1}^{m_n} E_{n,i}$,
	%\{ \Po_{\lambda(n)}(B(x_{n,i},r_n))  \geq  k(n) \}$,
	for all large enough $n$ we have
$$
	\Pr[E_n^c] \leq ( 1- \eta n^{8\eps-b/d})^{m_n} \leq \exp( - \eta m_n
	n^{8\eps -b/d} 
).
$$ 
By assumption $m_n = \nu(r_n , a r_n^d) = \Omega (r_n^{-b})$ so that 
for large enough $n$ we have $m_n \geq n^{(b/d) - \eps}$,
and therefore $\Pr[E^c_n]$ is
 is summable in $n$.

	 By Lemma \ref{lemChern}(d), 
and Taylor expansion of $H(x)$ about $x=1$ 
(see the print version of  \cite[Lemma 1.4]{RGG} for details; there may be a typo in the electronic
version),
for all $n $ large enough
$ \Pr[Z_{\lambda(n)} < n] \leq \exp( - \frac19  n^{1/2})$.
Similarly 
$ \Pr[Z_{\lambda^-(n)} > n] \leq \exp( - \frac19  n^{1/2})$.
%which is summable in $n$. 
If $E_n$ occurs, and $Z_{\lambda^-(n)} \leq n$, and $Z_{\la(n)} \geq n$,
then for some $ i \leq m_n $ there is at least one point
of $\X_n$ in $B(x_{n,i},\delta' r_n)$ and at most $k(n) $ points
of $\X_n$ in $B(x_{n,i},r_n)$, and hence $L_{n,k(n)}
> (1- \delta')r_n$. Hence
%Since $L_{m,k}$ is nonincreasing in $m$,
by the union bound
$$
\Pr[L_{n,k(n)} \leq r_n(1- \delta')] \leq
%\Pr[L_{Z_{\lambda(n)},k(n)} \leq r_n ]
\Pr[E_n^c]
%+ \Pr[ Z_{\lambda(n) } < n ]  \leq 
% \Pr[ E_n]+
+ \Pr[ Z_{\lambda(n)} < n] + 
 \Pr[ Z_{\lambda^-(n)} > n]  
,
$$
which is summable in $n$ by the preceding estimates. Therefore
by the Borel-Cantelli lemma,
\bean
 \Pr[  \liminf ( n  L_{n,k(n)}^d /\log n) \geq \nalpha (1-\delta')^d  ] 
 =1,
%~~~~~ \alpha < a^{-1} \hH_\beta(\gamma(a)/d),
~~~~~ \nalpha < a^{-1} \hH_\beta(b/d), \delta' \in (0,1),
\eean
so the result follows for this case too.
%\qed
\end{proof}

%Given  $r>0$, and $D \subset A$, define the `covering number' 
%\bea
%\kappa(D,r): = \min \{m  \in \N: \exists x_1,\ldots,x_m \in D ~{\rm with} ~ D \subset \cup_{i=1}^m B(x_i,r) \}.  \label{covnumdef}
%\eea

\subsection{{\bf Proof of  Theorem \ref{thmpolytope}}}
\label{secproofstrong}

In this subsection we assume, as in Theorem \ref{thmpolytope}, that
$A$ is a compact convex finite  polytope in $\R^d$.
We also assume that the probability measure $\mu$
has density $f$ with respect to Lebesgue measure on $\R^d$, and that  
$f|_A$ is continuous at $x$
for all $x \in \partial A$, and  that $f_0 >0$, recalling
 from (\ref{f0def}) that $f_0: = \essinf_{x \in A} f(x)$.
 Also we let $k(n)$ satisfy  (\ref{kcond})
for some $\beta \in [0,\infty]$.
Let $\Vol$ denote $d$-dimensional Lebesgue measure
%{\bf (maybe add $\beta = \infty$ later)}

%\subsection{{\bf Proof of Proposition \ref{thm1}}} 
%\label{subsecpfprop1}
 \begin{lemm}
	 \label{l:volLB}
	 There exists $\epsilon' >0$ depending only on $f_0$ and $A$,
	 such that
 \eqref{e:vol_lowB} holds.
 \end{lemm}
 \begin{proof}
	 Let $B_0$ be a (fixed) ball contained in $A$, and let $b$ denote
	 the radius of $B_0$. For $x \in A$, let 
	 $S_x$ denote the convex hull of $B_0 \cup \{x\}$.
	 Then $S_x \subset A$
	 since $A$ is convex. If $x \notin B_0$, then
	 for $r < b$ the set $B(x,r) \cap S_x$  is the intersection
	 of $B(x,r)$ with a cone having vertex $x$, and since
	 $A$ is bounded the angular volume of this cone is bounded
	 away from zero, uniformly over $x \in A \setminus B_0$.
	 Therefore $r^{-d} \Vol(B(x,r) \cap A)$ is bounded away from zero
	 uniformly over $r \in (0,b)$ and $x \in A \setminus B_0$
	 (and hence over $x \in A$). Since we assume $f_0 >0$,
	 (\ref{e:vol_lowB}) follows.
 \end{proof}

 Recall that $\nu(r,a)$ was defined just before Proposition \ref{gammalem}.
Recall  that for each face $\varphi \in \Phi^*(A)$
we denote the angular  volume of $A$ at $\varphi$ by
$\rho_{\varphi}$, and set $f_\varphi := \inf_{\varphi} f(\cdot)$ (if $\ph \in
\Phi(A)$) or $f_\ph = f_0 $ (if $\ph =A$).

\begin{lemm}
	\label{lemtopelb}
	Let $\varphi \in \Phi^*(A)$. Assume $f|_A$ is continuous at
	$x$ for all $x \in \varphi$.
Then, almost surely:
	\begin{align}
	\liminf_{n \to \infty} \left( n  L^d_{n,k(n)}/ k(n) \right)
	&
	\geq ( \rho_\varphi f_\varphi)^{-1}  ~~~ & {\rm if} ~ \beta = \infty; 
	\label{0704e3}
	\\
\liminf_{n \to \infty} \left( n  L^d_{n,k(n)}/ \log n \right)
		&
	\geq (\rho_\varphi f_{\varphi} )^{-1} \hH_\beta(D(\varphi)/d)  
~~~ & {\rm if} ~ \beta < \infty 
	.
	\label{0704f3}
	\end{align}
\end{lemm}
\begin{proof}
	Let $a > f_{\varphi}$.
	Take $x_0 \in \varphi$ 
	such that $f (x_0) <a$. If $D(\varphi) >0$, assume
	also that $x_0 \in \varphi^o$.
	By the assumed continuity of $f|_A$ at $x_0$, for all
	small enough $r >0$ we have
	$\mu(B(x_0,r)) \leq
	a  \rho_\varphi r^d$, so that
	$\nu(r,a \rho_\varphi r^d) = \Omega(1)$ as
	$r \downarrow 0$.
	Hence, by Proposition \ref{gammalem} (taking $b=0$), if  
	$\beta = \infty$ then almost surely   $\liminf_{n \to \infty} n L_{n,k(n)}^d/k(n) \geq 1/(a \rho_\varphi)$, and (\ref{0704e3}) follows.   
	
If $\beta < \infty$ and if $D(\varphi) =0$, then by Proposition \ref{gammalem}
	(with $b=0$), almost surely
	\linebreak
$\liminf_{n \to \infty} (n L_{n,k(n)}^d/\log n)
\geq \hH_\beta(0)/(a \rho_\varphi)$, and hence (\ref{0704f3}) in this case.

	Now suppose $\beta < \infty$ and $D(\varphi)>0$. Take
	$\delta >0 $ such that $f(x) < a$ for all
 $x \in B(x_0, 2 \delta) \cap A$, and such that moreover
	$B(x_0,2 \delta) \cap A = B(x_0,2 \delta) \cap (x_0+
	\cK_\varphi)$ (the cone $\cK_\varphi$ was defined
	in Section \ref{secLLN}).
	Then for all $x \in B(x_0,\delta) \cap \varphi$
	and all $r \in (0,\delta)$,
	we have $\mu(B(x,r)) \leq  a \rho_\varphi   r^d$.

There is a constant $c >0$ such that for small enough $r >0$ we can find at
least $cr^{-D(\varphi)}$ points $x_i \in B(x_0,\delta) \cap \varphi$ that
are all at a distance more than $2 r$ from each other, and therefore
$\nu( r,a \rho_\varphi r^d ) = \Omega( r^{-D(\varphi)})$ 
as $r \downarrow 0 $.  Thus by Proposition \ref{gammalem} we have
$$
	\liminf_{n \to \infty} \left( nL_{n,k(n)}^d/k(n) \right)
	\geq  ( a \rho_\varphi)^{-1} \hH_{\beta}(D(\varphi)/d),
$$
	almost surely, and (\ref{0704f3}) follows.
\end{proof}

 If we assumed $f|_A$ to be continuous  on all of $A$,
 we would not need the next lemma because we could instead use Lemma
 \ref{lemtopelb} for $\varphi =A$ as well as for lower-dimensional faces.
 However, in  Theorem \ref{thmpolytope}
 we make the weaker assumption that $f|_A$ is continuous at $x$ only
 for $x \in \partial A$. In this situation,
 we also require the following lemma to deal with $\ph = A$.

%\begin{lemm}
%\label{packlem}
%%Suppose  $B$ is Riemann measurable with $\mu(B) >0$.
%Let $\alpha > f_0$.
%	Then $\liminf_{r \downarrow 0} r^d \nu(r, \alpha \theta_d r^d) > 0$.
%\end{lemm}
%\begin{proof}
%	This is \cite[Lemma 6.4]{CovPaper}, taking $B=A$ here.
%\end{proof}

 %Recall now that we assume $k(n)$ satisfy (\ref{kcond}).
\begin{lemm} 
\label{lemliminf}
	%Suppose that $B$ is compact and Riemann measurable with $\mu(B) >0$,
	%and either $B \subset A^o$ or $B=A$.
	It is the case that
	\begin{align}
		\Pr[ \liminf(n  L_{n,k(n)}^d/k(n) ) & \geq 1
		/(\theta_d f_0)]=1
	%	~~~~~
		&{\rm if} ~\beta = \infty;
\label{0328a}
\\
 \Pr[  \liminf_{n\to\infty} 
	%( n \theta_d \tL_{n,k(n)}^d /\log n) \geq \hH_\beta(1)/f_0  ] =1.
		( n  L_{n,k(n)}^d /\log n) & \geq \hH_\beta(1)/(\theta_df_0)] =1
	%	~~~~~
		&{\rm if} ~\beta < \infty.
\label{liminf1}
	\end{align}
\end{lemm}
\begin{proof} 
Let $ \alpha > f_0$. Then
	by taking $B=A$ in 
	 \cite[Lemma 6.4]{CovPaper}, 
	\bea
	\label{e:packlem}
	\liminf_{r \downarrow 0} 
	r^d \nu(r, \alpha \theta_d r^d) > 0.
	\eea
	Set $r_n := (k(n)/(n \theta_d \alpha))^{1/d}$ if
$\beta = \infty$, and set
%\linebreak
$r_n:= (\hH_\beta(1) (\log n)/(n \theta_d \alpha))^{1/d}$ if $\beta < \infty$. 
%
%	Assume for now that $B$ is compact with $B \subset A^o$, so that there exists $\delta >0$ such that $B \subset A^{(\delta)}$.  Then, even if $f$ is not continuous on $A$, we can find $x_0 \in B$ with $f(x_0) < \alpha$, such that $x_0$ is a Lebesgue point of $f$.  Then for all small enough $r>0$ we have $\mu(B(x_0,r)) < \alpha \theta_d r^d$, so that $\nu(B,r,\alpha \theta r^d) = \Omega(1)$ as $r \downarrow 0$.
%

	If $\beta = \infty$, then by (\ref{e:packlem}) we
	can apply Proposition \ref{gammalem} (taking $a= \alpha \theta_d$
	and $b=0$) to deduce that 
	$\liminf_{n \to \infty} nL_{n,k(n)}^d/k(n) \geq (\theta_d \alpha)^{-1}$, almost surely,
	%Hence for all large enough $n$ we have $L_{n,k(n)} > r_n$;
	%provided  $n$ is also large enough  so that $r_n < \delta$ we also have $\tL_{n,k(n)} > r_n$, 
	and (\ref{0328a}) follows.

	Suppose instead that $\beta < \infty$.
	By (\ref{e:packlem}),
	$\nu(r, \alpha \theta_d r^{d}) = \Omega(r^{-d})$ as
	$r \downarrow 0$.
	Hence by Proposition \ref{gammalem}, almost surely
	$\liminf_{n \to \infty}
	\left( n L_{n,k(n)}^d /\log n \right) \geq (\alpha \theta_d)^{-1} 
	\hH_\beta(1)$. The result follows by letting $\alpha\downarrow f_0$. % and hence for large enough $n$ we have
	%$L_{n,k(n)} >  r_n $ 
	%and also 
	%$\tL_{n,k(n)} >  r_n $, 
	%which
	%yields (\ref{liminf1}). 
%
	%Finally, suppose instead that $B=A$.  Then by using e.g. \cite[Lemma 11.12]{RGG} we can find compact, Riemann measurable $B' \subset A^o$ with $\mu(B') >0$ and $\essinf_{x \in B'} f(x) < \alpha$.  Define $S_{n,k}$ to be the smallest $r \geq 0$ such that every point in $B'$ is covered at least $k$ times by balls of radius $r$ centred on points of $\X_n$.  By the argument already given we have almost surely for all large enough $n$ that $S_{n,k(n)} > r_n $ and also $B' \subset A^{(r_n)}$.  For such $n$, there exists $x \in B' \subset B \cap A^{(r_n)}$ with $\X_n(B(x,r_n)) < k(n)$, and hence by (\ref{Fnequiv}), $\tL_{n,k(n)} > r_n$, which gives us (\ref{0328a}) and (\ref{liminf1}) in this case too.
	%\qed
\end{proof}

%\begin{proof}[Proof of Proposition \ref{thm1}] Under either hypothesis ((i) or (ii)), it is immediate from Lemmas  \ref{lemliminf} and \ref{lemlimsup} that (\ref{0617a}) holds if $\beta = \infty$ and (\ref{0315a}) holds if $\beta < \infty$.  

%It follows that almost surely $\tL_{n,k(n)} \to 0$ as $n \to \infty$, and therefore if we are in Case (i) (with $B \subset A^o$) we have $\tL_{n,k(n)} = L_{n,k(n)}$ for all large enough $n$.  Therefore in this case (\ref{0617a}) (if $\beta =\infty$) or (\ref{0315a}) (if $\beta < \infty)$ still holds with $\tL_{n,k(n)}$ replaced by $L_{n,k(n)}$.  \end{proof}

%\subsection{{\bf Proof of  Theorem \ref{thm2}}}
%\label{secproofstrong}
%\allco

%Given any $x \in \R^d$ and nonempty $S \subset \R^d$
%we set  $\dist(x,S):= \inf_{y \in S} \|x-y\|$.

\begin{proof}[Proof of  Theorem \ref{thmpolytope}] 
	%Suppose $\beta = \infty$. We obtain from
	%Lemmas \ref{k
	%Let $u > \max \left( \frac{1}{\theta_d f_0}, \max_{\varphi \in \Phi(A)} \frac{1}{f_\varphi \rho_\varphi} \right)$.  Setting $r_n = (u k(n)/n)^{1/d}$, we have from Lemma \ref{lemedge2} that  $G_{n,k(n),r_n,\varphi}$ occurs only finitely often, a.s.,  for each $\varphi \in \Phi(A)$.  Hence by Lemma \ref{Flem}, $L_{n,k(n),1} \leq r_n$ for all large enough $n$, a.s. Hence, almost surely $ \limsup_{n \to \infty}  \left( n L_{n,k(n),1}^d/k(n) \right) \leq u.  $
%
%	Since $u > 1/(\theta_d f_0)$, by Proposition \ref{thm1} we also have $\limsup_{n \to \infty}  \left( n \tL_{n,k(n)}^d/k(n) \right) \leq u$, almost surely,    and hence by (\ref{mineq2}), almost surely 
	%$$
	%\limsup_{n\to \infty} \left( n L_{n,k(n)}^d/k(n) \right) \leq \max \left( (\theta_d  f_0)^{-1}, \max_{\varphi \in \Phi(A)} 1/(f_\varphi \rho_\varphi) \right).
	%$$
	%Moreover, by Lemma \ref{lemtopelb}, Proposition \ref{thm1} and (\ref{mineq2}) we also have that
	%$$
	%\liminf_{n\to \infty} \left( n L_{n,k(n)}^d/k(n) \right) \geq \max \left( (\theta_d f_0)^{-1}, \max_{\varphi \in \Phi(A)} 1/(f_\varphi \rho_\varphi) \right),
	%$$
	%and thus (\ref{0717a}).

	First suppose $\beta < \infty$.
	%Let $u > \max \left( \hH_\beta(1)/ (f_0\theta_d), \max_{\varphi \in \Phi(A)} \hH_\beta(D(\varphi)/d)/(f_\varphi \rho_\varphi) \right)$.  Set $r_n := (u (\log n)/n)^{1/d}$.  Given $\varphi \in \Phi(A)$, by Lemma \ref{lemedge2}  there exists $\eps >0$ such that, setting $k'(n):= \lfloor(\beta + \eps) \log n \rfloor$, we have $\Pr[G_{n,k'(n),r_n,\varphi} ] = O(n^{-\eps})$.  Hence  by Lemma \ref{Flem} and the union bound, 
	%$$
	%\Pr[ n L_{n,k'(n),1}^d/ \log n > u] = \Pr[L_{n,k'(n),1} > r_n] = O(n^{- \eps}).
	%$$
	%Thus by the subsequence trick (Lemma \ref{lemtrick} (a)), $ \limsup_{n \to \infty} \left( n L_{n,k(n),1}^d/ \log n \right) \leq u, $ almost surely.  Since $u > \hH_\beta(1)/ (f_0\theta_d)$,  and we take $B=A$ here, by Proposition \ref{thm1} we also have a.s. that $\limsup_{n \to \infty}  \left( n \tL_{n,k(n)}^d/\log n \right) \leq u$, and hence by (\ref{mineq2}), almost surely 
%
	It is clear from (\ref{Rnkdef}) and (\ref{oldRnkdef}) that
	$L_{n,k} \leq R_{n,k+1}$ for all $n,k$.  Also by
	(\ref{kcond}) we have $(k(n)+1)/\log n \to \beta$ as $n \to \infty$.
	Therefore using 
	\cite[Theorem 4.2]{CovPaper} 
	for the second
	inequality below, we obtain almost surely that
	\bea
	\limsup_{n\to \infty} \left( \frac{n L_{n,k(n)}^d}{\log n} \right)
	\leq 
	\limsup_{n\to \infty} \left( \frac{n R_{n,k(n)+1}^d}{
		\log n} \right)
	%\nonumber \\
	\leq 
%	\max \left( \frac{
%		\hH_\beta(1) }{  f_0\theta_d } 
%	 ,
	\max_{\varphi \in \Phi^*(A)} \left(
	\frac{ \hH_\beta(D(\varphi)/d)}{ f_\varphi \rho_\varphi}
%	\right)
	\right).  
\label{e:Lupper}
	\eea
Alternatively, this upper bound could be derived
using (\ref{e:LleM}) and the asymptotic upper bound on $M_n$
that we shall derive in the next section for the proof of Theorem
\ref{t:M}. 

By Lemmas \ref{lemliminf} and \ref{lemtopelb},
	we have a.s. that
	\bea
	\liminf_{n\to \infty} \left( n L_{n,k(n)}^d/\log n \right) \geq 
	%\max \left( \frac{\hH_\beta(1)}{\theta_d f_0}, 
	\max_{\varphi \in \Phi^*(A)} \left( \frac{ \hH_\beta(D(\varphi)/d) }{
 f_\varphi \rho_\varphi } \right),
 %\right),
\label{e:Llower}
	\eea
	and combining this with (\ref{e:Lupper}) yields (\ref{0717b}).

	Now suppose $\beta = \infty$. In this case,
	again using the inequality $L_{n,k} \leq R_{n,k+1}$ and
	\cite[Theorem 4.2]{CovPaper}, we obtain instead of   
(\ref{e:Lupper}) that a.s.
	\bea
	\limsup_{n\to \infty} \left( n L_{n,k(n)}^d/k(n) \right)
	%\leq 
	%\limsup_{n\to \infty} \left( n R_{n,k(n)+1}^d/\log n \right)
	%\nonumber \\
	\leq 
	%\max \left( \frac{ 1 }{  f_0\theta_d } ,
	\max_{\varphi \in \Phi^*(A)} \left(
	\frac{ 1}{ f_\varphi \rho_\varphi}
	%\right)
	\right).
	\label{e:upinf}
	\eea
Also by Lemmas \ref{lemliminf} and \ref{lemtopelb},
instead of (\ref{e:Llower}) 
	we have a.s. that
	\bean
	\liminf_{n\to \infty} \left( n L_{n,k(n)}^d/k(n) \right) \geq 
	%\max \left( \frac{1}{f_0 \theta_d }, 
	\max_{\varphi \in \Phi^*(A)} \left( \frac{ 1 }{
 f_\varphi \rho_\varphi } 
 %\right)
 \right),
	\eean
	and combining this with (\ref{e:upinf}) yields (\ref{0717a}).
\end{proof}

\subsection{A general upper bound}
%[{\bf comments on the assumption: 1. upper and lower bound on volume estimate is replaced by a measure $\mu_*$ doubling condition on the metric space (this addresses Mathew's long comment from the previous version ). 2 the lower bound on $\mu$ measure  of balls  is implied by condition on $\mu(G_r\setminus G)$ thus is removed.}]

In this subsection we present an asymptotic upper bound for $M_{n,k(n)}$.
As we did for the lower bound in Section \ref{subsecgenbds}, we
shall give our result (Proposition \ref{p:up} below) in a more general setting; 
we assume that $A$ is a general metric space endowed with two Borel measures
$ \mu$ and $\mu_*$ (possibly the same  measure, possibly not). 
Assume that $\mu$ is a probability measure and that $\mu_*$ is a
%measure $\mu$, and also with a Borel 
{\em doubling measure}, meaning that there is a constant $c_*$
(called a {\em doubling constant} for $\mu_*$)
such that
$\mu_*(B(x,2r))\le c_* \mu_*(B(x,r))$ 
for all $x \in A$ and $r >0$.
We shall require  further conditions on $A$: an ordering
condition (O), a condition on balls (B),
a  topological condition (T) and a geometrical condition (G) as follows:
\begin{itemize}
\item[(O)] There is a total ordering of the elements of $A$.
\item[(B)] For all $x \in A$ and $r>0$, the ball 
	%$B(x,r):= \{y \in A: \dist(x,y) \leq r\}$ is connected.
	$B(x,r)$ is connected.
\item[(T)] The space $A$ is unicoherent (see \cite[Section 9.1]{RGG}), and
also connected.
\item[(G)]
There exists $\delta_1>0$, and $K_0 \in(1,\infty)$,
 such that for all $r< \delta_1$ and any $x\in A$, the number of components of  $A\setminus B(x,r)$ is at most two,
 and if there are two components, at least one of these components
has diameter at most $K_0 r$.
\end{itemize}
Given $D \subset A$ and $r >0$, we write $D_r$ for $\{y \in A: \dist(y,D) \leq r\}$.
Also, let $\kappa(D,r)$ be the $r$-covering number of
 $D$, that is, the minimal $m\in \N$ such that $D$ can be
 covered by $m$ balls centred in $D$ with radius $r$.

 As before, given $\mu$ we assume $X_1,X_2,\ldots$ to  be independent
 $\mu$-distributed random elements of $A$ with the $k$-connectivity
 threshold $M_{n,k}$ defined to be the minimal $r$ such that
 $G(\X_n,r)$ is $k$-connected, with $\X_n := \{X_1,\ldots,X_n\}$.
  
\begin{prop}[General upper bound]
\label{p:up}
Suppose that $(A,\mu,\mu_*)$ are as described above and $A$
	satisfies conditions (O), (B), (T), (G).  
	Let $\ell\in \N$ and let $d >0$. For each $j \in [\ell]$ let
 $a_j >0,   b_j\ge 0$. Suppose that for each $K \in \N$,
	there exists $r_0(K)>0$ such that for 
	 all $r\in (0,r_0(K))$,  there is a partition
	$\{T(j,K,r), j\in[\ell]\}$ of $A$ with the following two properties.
	Firstly for  each fixed $K \in \N$, $j\in[\ell]$, we have
	\begin{align}
		\kappa(T(j,K,r),r)=O(r^{-b_j}) ~~~{\rm as} ~~ r\downarrow 0,
		\label{1231a}
	\end{align}
and secondly, for all $K \in \N,$ $ j \in [\ell]$,  $r \in (0,r_0(K))$  and
	any $G\subset A$ intersecting $T(j,K,r)$ with 
$\diam(G)\le K r$, we have 
\begin{align}\label{0923}
\mu(G_r\setminus G) \ge a_j r^d.
\end{align}
Assume \eqref{kcond}. 
	Then, almost surely, 
	\begin{align*}
		\limsup_{n\to\infty} \big( n M^d_{n,k(n)}/k(n) \big) &\le 
		\max_{j \in [\ell]} (a_j^{-1})
		%~~~~~~~~~
	 &{\rm if} ~\beta = \infty;
		\\
		\limsup_{n\to\infty} \big( n M_{n,k(n)}^d/\log n \big)& \le \max_{j\in[\ell]}
		(a_j^{-1} \hat H_\beta(b_j/d)) 
		%~~~~~~~~~~
		 &{\rm if} ~\beta<\infty.
		\end{align*}
\end{prop}

Later we shall use Proposition \ref{p:up} 
in the case where $A$ is a convex polytope in $\R^d$
to prove Theorem \ref{t:M}, taking  $\mu$ to be the measure
with density $f$ and taking $\mu_*$ to be the restriction of
Lebesgue measure to $A$
(in fact, if $f$ is bounded above then we could take $\mu_*= \mu$ 
instead).
The sets in the partition each represent a  region near to a particular
face   $ \ph \in \Phi^*(A)$ (if $\ph= A $
the corresponding set in the partition is
 an interior region).
In this case, coefficients $a_j$ in 
 the measure lower bound \eqref{0923} depend heavily on the geometry of the determining cone near a particular face.

 As a first step towards proving Proposition \ref{p:up},
we spell out some useful consequences of the measure doubling property.  
In this result (and again later) we use $|\cdot|$ to denote
the cardinality (number of elements) of a set.

 \begin{lemm}
 \label{l:cov}
 Let $\mu_*$ be a doubling measure on the metric space $A$, with
	 doubling constant $c_*$.  We have the following. 

	 (i) For any $\epsilon\in(0,1)$, there exists $\rho(\eps)\in\N$
	 such that  $\kappa(B(x,r), \epsilon r) \le \rho(\eps)$ for all
	 $x\in A, r\in (0,\infty)$.
 
 (ii)
	 For all $r \in (0,1)$ and
	 all $D \subset A$,
	 %with $\diam(D) < \infty$,
	 we can find $\cL \subset D$ with $|\cL| \leq \kappa(D,r/5)$,
	 such that $D \subset \cup_{x \in \cL} B(x,r)$, and moreover
	  the balls $B(x,r/5)$, $x \in \cL$, are disjoint.
 \end{lemm}
\begin{proof}
To prove (i), let $x \in A, r >0$.  By the  Vitali covering lemma, we can find
a set $\mathcal U\subset B(x,r)$ such that balls $B(y,\epsilon r/5),
y\in \mathcal U$ are disjoint and that 
$ B(x,r) \subset \cup_{y\in \mathcal U} B(y,\epsilon r) $.
Set $\rho(\eps) := \lceil c_*^{\lceil \log_2 (15/\epsilon) \rceil}\rceil$.
Then by using the doubling property of $\mu_*$ repeatedly, we
have $\mu_*(B(y,3r)) \leq \rho(\eps) \mu_*(B(y,r/5))$ for all
	$y \in A$. 
	Moreover $B(x,2r)\subset B(y,3r)$ for all $y\in \mathcal U$.
	Also $\cup_{y\in \mathcal U} B(y,\eps r/5) \subset B(x, 2r)$
	and the union is disjoint.  Thus
\begin{align*}
	| \mathcal U| \mu_*(B(x,2r)) \leq
	\sum_{y\in \mathcal U} \mu_*(B(y,3r)) \le \rho(\eps) \sum_{y\in\mathcal U} \mu_*(B(y,\epsilon r/5)) \le \rho(\eps) \mu_*(B(x,2r)),
\end{align*}
	and therefore $|\mathcal U|\le \rho(\eps)$; the claim about $\kappa(B(x,r), \epsilon r)$ follows. 

	Now we prove (ii).
	%By the assumption $\diam(D) < \infty$, and Part (i),
	%we have $\kappa(D,r/3) < \infty$.
Let $\cL^0 \subset D$ with  $|\cL^0| = \kappa(D,r/5)$ and with
	$B \subset \cup_{x \in \cL} B(x,r/5)$. 
 By the Vitali covering lemma, we can find $\cL \subset \cL^0$
	such that $D\subset \cup_{x \in \cL} B(x,r)$
	and the balls $B(x,r/5), x \in \cL,$ are disjoint, and  (ii) follows.
\end{proof}

 Given countable $\sigma \subset A$,
$r >0$ and $k \in \N$,
%
%Given a countable subset $\mathcal L$ of $A$,
 %we say that $\sigma\subset \mathcal L$ is $r$-connected if $\sigma\oplus B(o,r/2)$ is connected subset of $A$, or equivalently, $G(\sigma, r)$ is connected.
 we say that $\sigma$ is $(r,k)$-connected if 
 the geometric graph $G(\sigma, r)$ is $k$-connected. Assuming 
  condition (B) holds, we see that $\sigma$ is $(r,1)$ connected if and only
if $\sigma_{r/2}$ is a
 connected  subset of $A$.

\begin{lemm}[Peierls argument]  
	Assume (O).
	Let $\ell \in \N$, $a \in [1,\infty) $.
Let $r \in (0,1/a)$ and $n \in \N$.  Let $\mathcal L \subset A$
	  with the property that $|\cL \cap B(x,r)| \leq \ell$
	  for all $x \in A$, and let $x_0 \in \cL_r$.
	  Then the number of $(ar,1)$-connected subsets of
 $\mathcal L$ containing  $x_0$ with cardinality $n$ is at most 
	$c^n$, where $c$ depends only on $\ell$,  $a$ and $c_*$.
\label{l:pa}
\end{lemm}

\begin{proof}
	First we claim that $|\cL \cap B(x,ar) | \leq \ell \rho(1/a)$ for
	all $x \in A$,
	where $\rho(1/a) $ is as given in Lemma \ref{l:cov}-(i).
	Indeed, we can cover $B(x,ar)$ by $\rho(1/a)$ balls 
	of radius $r$, and each of these balls contains at most $\ell $ points
	of $\cL$.

There is a standard algorithm (of constructing a non-decreasing sequence
	of lists) for counting the connected sets of $\Z^d$;
	see \cite[Lemma 9.3]{RGG} for details of the algorithm. 

	The algorithm remains valid in this general setting,
 with the lexicographical ordering replaced by the total ordering of $A$ (using assumption (O)).
 This algorithm has to stop at time $n$ (cardinality of the set), and 
at each step the number of  possibilities for the set of the added elements
	is bounded by $2^{\ell \rho(1/a)}$ (all possible subsets
	of the set of points of $\mathcal L$ 
	within distance $a r$ from a fixed point);
	%which is at most $\rho_1$),
	hence the number of $ar$-connected sets of cardinality 
	$n$ is at most $2^{\ell \rho(1/a) n}$.
\end{proof}

%and the third imposes a net structure akin to regular lattices in $\R^d$. 

%{\bf move some of the above to the intro if we do $M$}

%By the , there is no loss of generality in the first half of (G)-(ii). 

Preparing for a proof of Proposition \ref{p:up}, we recall a 
condition that is equivalent to $k$-connectedness  of a graph $G$.
 We say that non-empty sets $U,W\subset V$ in a graph 
 $G$ with vertex set $V$ form a 
{\em$k$-separating pair} if (i) the subgraph of $G$ induced by $U$ is connected, and likewise for $W$; (ii) no element of $U$ is adjacent to any element of $W$; (iii) the number of vertices of $V\setminus (U\cup W)$ lying adjacent to $U\cup W$ is at most $k$. We say that $U$ is
a {\em  $k$-separating set} for $G$ if
 (i) the subgraph of $G$ induced by $U$ is connected,
 and (ii) at most  $k$  vertices of $V \setminus U$ lie adjacent
to $U$. 
   The relevance of these definitions is presented in the following lemma.

  \begin{lemm}\cite[Lemma 13.1]{RGG}
  \label{l:sep}
  Let $G$ be a graph with more than $k+1$ vertices. Then $G$ is either $(k+1)$-connected, or it has $k$ separating pair, but not both. 
  \end{lemm}

By Lemma \ref{l:sep}, to prove Proposition \ref{p:up}
it suffices to prove,
for arbitrary $u> \linebreak \max_j a_j^{-1} \hat H_\beta(b_j/d)$,
 the non-existence of $(k(n)-1)$-separating pairs in
 $G(\X_n,r_n)$ with $r_n = (u\log n/n)^{1/d}$, 
 as $n\to\infty$. Notice that, for any fixed $K\in \N$, if $(U,W)$ is a $(k-1)$-separating pair, then either both $U$ and $W$ have
 diameter at least $K r_n$, or one of them, say 
$U$, is a $(k-1)$-separating set 
%(i) holds for $U$ (iii) holds with $U$ in place of $U\cup W$) 
of diameter at most $Kr_n$. 
Here by the {\em diameter} of a a non-empty set $U \subset A$
we mean the number $\diam(U):= \sup_{u,v \in U}\dist(u,v)$.

The goal is to prove that neither outcome is possible when $n\to\infty$.   
Let us first eliminate the existence  of a small separating set.

\begin{lemm}
\label{l:small}
Suppose the assumptions of Proposition \ref{p:up} hold.
If $\beta=\infty$, let $u>\max_j a_j^{-1}$ and for $n\in\N$, set $r_n = (u k(n)/n)^{1/d}$. If $\beta<\infty$,  let $u> \max_{j\in [\ell]} a_j^{-1} \hat H_\beta(b_j/d)$, and for $n \in \N$ 
	set $r_n = (u (\log n)/n)^{1/d}$.
For $K \in \N$, let 
	$E_n(K,u)$ be the event that there exists a $(k(n)-1)$-separating set
for $G(\X_n,r_n)$  of diameter at most $K r_n$. Then, given 
any $K \in \N$, almost surely $E_n(K,u)$ occurs for only finitely many $n$.   
\end{lemm}

%[{\bf proof revised, using $\mu_*$ for counting related arguments, case infinite beta added}]

\begin{proof}
	First assume $\beta<\infty$. 
%For this proof we need some notation.
% Let $j \in [\ell]$.
The condition on $u$ implies that
	$ua_j > \beta$ and $ua_j H(\beta/(ua_j)) > b_j/d$,
for each $j \in [\ell]$.  Then we can and do choose $ \beta' > \beta $ 
and $\epsilon \in (0,1/4)$ such  that for each $j \in [\ell]$, 
	$(1- 3 \eps)^du a_j > \beta'$ and
%$\beta<\beta'<\beta''<\hat H_\beta(b_j/d)<ua_j$.
\begin{align*}
(1-3\epsilon)^d ua_j H\Big(\frac{\beta'}{(1-3\epsilon)^d ua_j}\Big)>\frac{b_j}{d} +\epsilon. 
\end{align*}
For $n \in \N$
 define $k'(n)=\lceil \beta' \log n\rceil$.
	%and $k''(n)=\lceil \beta'' \log n\rceil$. 

	Let $K \in \N$, and for $r \in (0,r_0(K))$
	let $T(j,K,r)$ be as in the assumptions of Proposition \ref{p:up}.
For $j \in [\ell]$, we claim that
	$\kappa(T(j,K,r_n),\eps r_n/5)= O(r_n^{-b_j})$ as $n \to \infty$. 
%{\bf Please can you justify this?}
Indeed,
$$
	\kappa(T(j,K,r_n),\eps r_n/5) \le \kappa(T(j,K,r_n),r_n)\sup_{x\in A}\kappa(B(x,r_n), \epsilon r_n/5) \le \rho \kappa(T(j,K,r_n),r_n),
$$
where $\rho = \rho(\epsilon/5)$ is the constant in Lemma \ref{l:cov}-(i).
	The claim follows from  the assumption (\ref{1231a}).

	Choose $n_0 \in \N$ such that $r_n < r_0(k)$ for all
	$n \in \N$ with $n \geq n_0$.
By Lemma \ref{l:cov}-(i), for each $j \in [\ell]$ and $n \in \N$
 we can find a set
 $\mathcal L_{n}^{j}
	\subset T(j,K,r_n)$,
with $|\cL_n^j| \leq \kappa( T(j,K,r_n),\epsilon r_n/5) = O(r_n^{-b_j})$, 
such that  $T(j,K,r_n) \subset \cup_{x \in \cL_n^j}
 B(x, \epsilon r_n)$
 and that the balls 
	$B(x, r_n \epsilon /5)$, $x \in \cL_n^j$, are disjoint.
Set 
\begin{align}
\label{e:def_Ln}
 \mathcal L_n := \cup_{i=1}^\ell \mathcal L^j_n. 
\end{align}
%Hence $|\mathcal L^j_n|=O(r_n^{-b_j})$ for all $j\in [\ell]$ as $n\to\infty$. 

	For $n \geq n_0, j \in [\ell]$
	let $\mathcal T_n^j=\{\sigma\subset \mathcal L_n: \diam(\sigma)\le 2K r_n,  \sigma\cap T(j,K,r_n)\ne\emptyset \}$.   
	We claim that the cardinality of $\mathcal T^j_n$ is $O(|\cL_n^j|)=O(r_n^{-b_j})$. Indeed,  $\sigma\cap T(j,K,r_n) \ne\emptyset$ means
$\sigma\cap \mathcal L^j_n\ne\emptyset$. Moreover, as explained below,
\begin{align}
	\limsup_{n \to \infty} \sup_{x \in \cL_n} |B(x,2Kr_n) \cap \cL_n| < \infty,
	\label{1231b}
\end{align}
and $\diam(\sigma)\le 2Kr_n$. The claim about  cardinality follows from this.   

Now we show (\ref{1231b}). By Lemma \ref{l:cov}-(i), for
$n$ large and for all $x \in A$, we can cover $B(x,2Kr_n)$ by
$\rho(\eps/(10K)) $ balls of radius $r_n \eps/5$, and each of these
balls contains at most $\ell$ points of $\cL_n$.

	For $n \geq n_0$ and $\sigma \subset \mathcal L_n$, set 
\begin{align}
	D_{\sigma,n} := \sigma_{(1-2 \eps) r_n} \setminus \sigma_{\eps r_n}. 
	\label{e:Ds}
\end{align}
Let $J \in \N$ with $J >1/\epsilon$. For $m \in \N$,  define
	$z(m):=m^J$. For 
	%$\sigma\in \mathcal T^j_{z(m)}$ with  $j\in [\ell]$, define 
	$\sigma \subset \mathcal L_{z(m)}$, define 
\begin{align*}
F_m(\sigma) = \{  \X_{z(m)}( D_{\sigma, z(m)})< k'(z(m)) \}.
\end{align*}

	Now let $n \in \N$ and choose $m= m(n) $ such that $z(m ) \leq n 
	< z(m+1)$. Assume $z(m) \geq n_0$. 
Suppose that $E_n(K,u)$ occurs and let $U$ be
a $(k(n)-1)$-separating set of $G(\X_n,r_n)$ with
	$\diam(U)\le Kr_n$. We define its `pixel version' $\sigma(U) :=
	\mathcal L_{z(m(n))} \cap U_{\eps r_{z(m(n))}}$.
	%\{x\in \mathcal L_n: U\cap B(x,\epsilon r_n)\ne\emptyset\}$. 

	Since $\sigma(U)\subset A$, there exists $j\in [\ell]$ such that
	$\sigma(U)\cap T(j,K,r_{z(m(n))})
	\ne\emptyset$. By our choice of $\epsilon$, provided $n$ is large enough
	we have $\diam(\sigma(U))\le 2Kr_{z(m(n))}$. Therefore
	$\sigma(U) \in \cup_{j=1}^{[\ell]} \mathcal T^j_{z(m(n))}$.

	Since $U$ is $(k(n)-1)$-separating for $G(\X_n,r_n)$,
	we have 
	%$\X_n( (U\oplus B(o,r_n)) \setminus U )< k(n)$. 
	$\X_n( U_{r_n} \setminus U )< k(n)$. 
We claim that $\X_n(D_{\sigma(U),z(m(n))})<k(n)$ provided $n$ is large enough. 
Indeed, by the triangle inequality
$\sigma(U)_{(1-2 \eps)r_{z(m(n))}} \subset U_{(1-\eps)r_{z(m(n))}}
\subset U_{r_n}$ (for $n$ large),
while $U \subset \sigma(U)_{\eps r_{z(n(m))}}$.
Thus $D_{\sigma(U), z(m(n))} \subset U_{r_n} \setminus U$, and
the claim follows.
Also, provided $n$ is large enough, we have $k(n)\le k'(z(m(n)))$.
Thus we have the event inclusions
\begin{align*}
%\label{0922a}
	E_n(K,u) 
	& \subset 
	\cup_{j=1}^{\ell} \cup_{\sigma\in \mathcal T^j_{z(m(n))}}
	\{ \X_n(D_{\sigma,z(m(n))})< k(n)\}
	\\
	& \subset 
	\cup_{j=1}^{\ell} \cup_{\sigma\in \mathcal T^j_{z(m(n))}}
	F_{ m(n)}(\sigma).
\end{align*}

By (\ref{e:Ds}), for any $n \in \N$ and $\sigma \subset \mathcal L_n$ we have
$D_{\sigma,n} \supset (\sigma_{\eps r_n})_{(1-3 \eps)r_n} \setminus 
\sigma_{\eps r_n}$. Hence 
%Recall that $k'(n)=\lceil \beta' \log n\rceil$. 
by \eqref{0923}, for all large enough $n$ and all
$\sigma \in \cup_{j\in [\ell]} \mathcal T_{n}^j$
we have  $\mu (D_{\sigma,n})
%_{(1-2\epsilon)r_n} \setminus D_{\sigma,n})
\ge a_j (1-3\epsilon)^d r_n^d$.
A simple coupling shows that, provided $m$ is large,
we have
\begin{align*}
\Pr[\cup_{j\in[\ell]}\cup_{\sigma\in \mathcal T^j_{z(m)}} F_m(\sigma)]= \sum_{j=1}^{\ell} O(r_{z(m)}^{-b_j})\Pr[ \Bin(z(m),  (1-3\epsilon)^d a_j r_{z(m)}^d)
	< k'(z(m))].
\end{align*}
By Lemma \ref{lemChern}(b) and our choice of $r_n$ and $\epsilon$, provided $m$ 
is large, we have
\begin{align*}
\Pr[\cup_{j\in[\ell]}\cup_{\sigma\in \mathcal T^j_{z(m)}} F_m(\sigma)]\\
=O(1) \sum_{j=1}^{\ell}
	\exp\Big( (b_j/d)\log z(m) - (1-3\epsilon)^d u a_j  H\big(\frac{\beta'}{(1-3\epsilon)^d ua_j}\big)\log z(m) \Big) = O(m^{-J\epsilon}),
\end{align*}
which is summable in $m$. 

It follows from the Borel-Cantelli lemma that almost surely $\cup_{j\in[\ell]}\cup_{\sigma\in \mathcal T^j_{z(m)}} F_m(\sigma)$ occurs only for finitely many $m$ which implies that $E_n(K,u)$ occurs for only finitely many $n$. This completes the proof of the case $\beta<\infty$. 

Now assume $\beta=\infty$.  For the rest of the proof  assume also that $\epsilon\in (0,1)$ is such that $ua_j(1-\epsilon)^d>1$ for all $j \in [\ell]$.
We do not have to go through the subsequence argument as before because the growth of $k(n)$ is super-logarithmic. 
Now redefine $F_n(\sigma) := \{\mathcal X_n (D_{\sigma, n})< k(n)\}$. If $E_n(K,u)$ happens then we now redefine the pixel version of the separating
set $U$ as
$$
\sigma(U) := \cL_n \cap U_{\eps r_n},
$$
and enumerate the possible shapes $\sigma$ of the pixel version.  Thus we have 
\begin{align*}
	E_n(K,u) \subset \cup_{j=1}^\ell \cup_{\sigma\in \mathcal T^j_n} F_n(\sigma).
\end{align*}
Using estimates of $|\mathcal T^j_n|$,
%and a simple coupling argument, 
we have
\begin{align*}
	\Pr[E_n(K,u)] = \sum_{j=1}^\ell O(r_n^{-b_j}) \Pr[\Bin(n, (1-3\epsilon)^d a_j r_n^d )<k(n)].
\end{align*}
Noticing  $r_n^{-1}=O(n^{1/d})$, and applying Lemma \ref{lemChern}-(b) leads to 
\begin{align*}
	\Pr[E_n(K,u)] = O(n^{b_j/d}) \sum_{j=1}^\ell \exp\Big( -  (1-3\epsilon)^d a_j u k(n) H\big(\frac{k(n)}{ (1-3\epsilon)^d a_j  u k(n)}\big) \Big)
\end{align*}
which is summable in $n$, and the claim follows by the Borel-Cantelli lemma. 
\end{proof}

The following lemma eliminates  the existence of a $(k(n)-1)$-separating pair
with both diameters larger than $Kr_n$.

%{\bf [case infinite beta added]}
\begin{lemm}
\label{l:big}
	Let the assumptions of Proposition \ref{p:up} hold.
	If $\beta=\infty$, let $u>\max_j a_j^{-1}$ and for $n\in\N$, set $r_n = (u k(n)/n)^{1/d}$. If $\beta<\infty$,  let $u> \max_{j\in [\ell]} a_j^{-1} \hat H_\beta(b_j/d)$, and for $n \in \N$ 
	set $r_n = (u (\log n)/n)^{1/d}$.
%
%	 Set
%	$r_n = (u(\log n)/n)^{1/d}$ with $u> a_j^{-1} \hat H_\beta(b_j/d)$ for all $j\in[\ell]$.  
	For $K \in \N$
let $H_n(K,u)$ denote the event that there exists a $(k(n)-1)$-separating
	pair $(U,W)$ in $G(\X_n,r_n)$ such that 
	$\min(\diam(U),\diam(W))\ge K r_n$.  
	Then there exists $K_1 \in  \N$ such that almost surely
	$H_n(K_1,u)$ occurs for only finitely many $n$.
\end{lemm}

\begin{proof}
Suppose $H_n(K,u)$ holds.  
%that $(U,W)$ is a $k(n)-1$ separating pair of $G(\X_n, r_n)$ satisfying the diameter requirement. 
Then $U_{r_n/2}$ and $W_{r_n/2}$ are disjoint and connected in $A$.
One of the components of $A\setminus U_{r_n/2}$
contains $W$, denoted by $W'$. Set $U'= A\setminus W'$.
Then $U\subset U'$, $W\subset W'$ and $A= W'\cup U'$.
Let $\partial_W U := \overline{W'}\cap \overline{U'}$. Then
$\partial_WU$ is connected by the unicoherence of $A$. Moreover, 
any continuous path in $A$ connecting $U$ and $W$ 
	must pass through $\partial_W U$. 

Recall $ \delta_1$ and $K_0$ in the assumption (G).  We claim
	(and show in the next few paragraphs) that 
\begin{align}
\label{0922b}
\diam(\partial_W U) \ge \frac{1}{2K_0+2} \min(\delta_1/3,
	\diam(W)/3, \diam(U)/3). 
\end{align}
	Suppose the opposite. Setting $b=\diam(\partial_W U)$, we can find $x\in A$ such that $\partial_W U\subset B(x,b)$, and we can find $X\in U\setminus B(x,b), Y\in W\setminus B(x,b)$. Since $b< \delta_1/3$, the number of components of $A\setminus B(x,b)$ is at most two.
	There have to be two components because otherwise $X$ and $Y$ can be connected by a path in $A$ disjoint from $\partial U$, which is a contradiction.  

Suppose that $X$ lies in the  component of $A \setminus B(x,b)$ 
	having diameter at most $K_0 b$, denoted by $Q_X$, and $Y$ 
	lies in the other component, denoted by $Q_Y$ (if it is the other
	way round we reverse the roles of $X$ and $Y$ in the rest of this
	argument). 
We claim that there exists $X'\in U$ such that $\dist(X,X')>(2K_0+2)b$. If not, then for any $X_1, X_2\in U$, we have by triangle inequality that $\dist(X_1, X_2)\le 2(2K_0+2)b$, yielding that $\diam(U)\le 2(2K_0+2)b$, contradicting $\diam(U)>3(2K_0+2)b$ by the negation of \eqref{0922b}.  

We claim that $\dist(X,B(x,b)) \leq K_0b$. To see this, using the assumed
connectivity of $A$, take a continuous path in $A$ from  $X$ to $Y$.
The first exit point of this path from $Q_X$ lies in $B(x,b)$ (else it would
not be an exit point from $Q_X$) but also in the closure of $Q_X$, and hence
in $B(X,K_0b)$.  This yields the latest claim.

We show that $X'$ and $Y$ have to be in the same component of 
$A \setminus B(x,b)$.  To this end, notice first that $X'$ cannot be in
$Q_X$,  because for any $z\in Q_X$, 
\begin{align*}
\dist(X,z)\le K_0 b <  (2K_0+2) b.
\end{align*}
Secondly, $X'$ cannot be in $B(x,b)$ either because for any $z\in B(x,b)$, we have
\begin{align*}
\dist(z,X)\le \dist(X,B(x,b)) + 2b \le (K_0 +2)b < (2K_0+2) b.
\end{align*}
Therefore, $X'$ has to be in $Q_Y$, and we reach again to a contradiction that $X'$ and $Y$ can be connected by a path in $A$ disjoint from $\partial U$. We have thus proved \eqref{0922b}.

Let $\epsilon\in (0,1/9)$ and let 
	$\mathcal L_n$ be as defined at \eqref{e:def_Ln}
	(the $\eps$ does not have to be the same as it was there). 
	Recall that $\mathcal L_n$ has the {\em covering property}
	that for every $x \in A$ we have
	$\LL_n \cap B(x,r_n \eps) \neq \emptyset$
	and the {\em spacing property}
	that $|\LL_n \cap B(x,r_n \eps/3)| \leq \ell$ for all such $x$.

Define $D_WU = \{ x\in \mathcal L_n:
	B(x,\epsilon r_n)\cap \partial_WU \ne\emptyset  \}$. 
	Then by the covering property of $\mathcal L_n$,
	$(D_WU)_{\epsilon r_n}$ is connected and covers $\partial_WU$.
	That is,
	$D_WU$,  as a  subset of the metric space $A$, is
	$(2\epsilon r_n,1)$-connected.

 By \eqref{0922b} and the occurrence of $H_n(K,u)$,  we have   
\begin{align*}
	2\epsilon r_n |D_WU| 
	\ge \diam(\partial_W U) \ge  \min(\delta_1/3, Kr_n/3)/(2K_0+2)
\end{align*}
Therefore, provided $n$ is large, we have $|D_WU|\ge K/(6\epsilon(2K_0+2))$. 

We claim that there is a constant $c \in  (0,\infty)$, independent
of $n$, such that for
all $q \in \N$,
if $|D_WU|=q$ then $D_WU$ can take at most $O(r_n^{-\max(b_j)} c^q)$ 
possible 'shapes'.  Indeed, given $x_0 \in \LL_n$, set 
$$\mathcal U_{n,q}(x_0):=\{\sigma\subset \mathcal L_n: |\sigma|=q, \sigma \mbox{ is $(2\epsilon r_n,1)$-connected}, x_0\in \sigma\}.$$
Then $D_W U \in \cup_{j\in[\ell]} \cup_{x_0\in T(j,K,r_n)\cap 
\mathcal L_n} \mathcal U_{n,q}(x_0)$. 
By Lemma \ref{l:pa},  we have $|\mathcal U_{n,q}(x_0)|\le c^{q}$ for some finite constant $c$.  Recall from the proof of Lemma \ref{l:small} that 
$|T(j,K,r_n)\cap \mathcal L_n| = O(r_n^{-b_j})$.  The claim follows.

 For all $n \in \N$, if $x \in \partial_WU$ then $\dist(x,U) = r_n/2$.
 Therefore by the triangle inequality,
$(D_WU)_{\epsilon r_n/5} \subset U_r$, while   $U\cap (D_WU)_{\epsilon r_n/5}
=\emptyset$; 
hence $\X_n \cap (D_W U)_{\eps r_n/5} =\emptyset$.
This, together with the the union bound, yields that 
\begin{align}\label{0922c}
\Pr[ H_n(K,u)]\le \sum_{q\ge K/(6\epsilon(2K_0+2))} \sum_{\sigma}
	\Pr[ \X_n(\sigma_{\epsilon r_n/5}) < k(n) ],
\end{align} 
where the second sum is over all possible shapes
$\sigma \subset \mathcal L_n$ of cardinality $q$ that are
$(2 \epsilon r_n,1)$-connected.  Since every point in $A$ is covered at most
$\ell$ times, by \eqref{0923} (with $G=\{z\}$), there exists $\epsilon_1\in(0,1)$ such that 
\begin{align*}
	\mu(\sigma_{\epsilon r_n/5}) \ge
	(1/\ell) \sum_{z\in\sigma}\mu(B(z,\epsilon r_n/5))  
	\ge (q/\ell) \epsilon_1 (\epsilon r_n/5)^d.
\end{align*}

Suppose $\beta<\infty$. 
Set $\eps_2 := (\eps_1/\ell) (\eps/5)^d$.
By
\eqref{0922c}
%a simple coupling, 
and
Lemma \ref{lemChern}(b),
 provided $n$ is large,
\begin{align*}
\Pr[H_n(K,u)]\le \sum_{q\ge K/(6\epsilon(2K_0+2))} O(r_n^{-\max(b_j)} c^q)
	\Pr[\Bin(n,  \eps_2q r_n^d)< (\beta+1) \log n]
	\\
=O(1) \sum_{q\ge K/(6\epsilon(2K_0+2))} c^q
	\exp\Big( (\max(b_j)/d) \log n - \epsilon_2 q 
	u  H\big(\frac{\beta+1}{\eps_2 q 
	u}\big) \log n\Big).
\end{align*}
By the continuity of $H(\cdot)$ and the fact that
$H(0)=1$,
there exists $q_0>16  /(\epsilon_2 u)$ such that for any $q>q_0$,  we have
$H\big(\frac{\beta+1}{q \epsilon_2  u}\big)>1/2$ and 
$q u \epsilon_2 >4 \max(b_j)/d$. Choosing $K = 6\epsilon 
(2K_0+2)q_0$ so that $q\ge q_0$ in the sum,
we see that the exponent of the exponential is bounded above by
\begin{align*}
(\max(b_j)/d) \log n  - q \epsilon_2 (u/2)  \log n  \le
	-(qu \epsilon_2/4) \log n.
\end{align*}
Therefore, we have for $n$ large that
\begin{align*}
	\Pr[H_n(K,u)] &= O(1) \sum_{q\ge q_0} c^q 
	\exp( - q u(\epsilon_2/4) \log n) \\
	& = O(1) \sum_{q\ge q_0}	\exp( - q u(\epsilon_2/8) \log n)
= O(\exp( - q_0 u(\epsilon_2/8) \log n)) =O(n^{-2}). 
\end{align*}
The result in this case follows by applying the Borel-Cantelli lemma. 

If $\beta=\infty$, then by
\eqref{0922c} and the estimates of $|\cup_j\cup_{x_0}\mathcal U_{n,q}(x_0)|$ as previously,  we have
\begin{align*}
\Pr[H_n(K,u)]\le \sum_{q\ge K/(6\epsilon(2K_0+2))} O(r_n^{-\max(b_j)} c^q) \Pr[\Bin(n,  (q/\ell) \epsilon_1 (\epsilon r_n/5)^d)< k(n)]. 
\end{align*}
We  have $ r_n^{-\max(b_j)} = O(n^{\max(b_j)/d})$, and by  Lemma \ref{lemChern}-(b),
\begin{align*}
	\Pr[H_n(K,u)] 
	%\\
	\le \sum_{q\ge K/(6\epsilon(2K_0+2))}  c^q
	\exp\Big( (\max(b_j)/d) \log n - 
	q \epsilon_2 u k(n) H(\frac{k(n)}{q \epsilon_2  u k(n)}) \Big).
\end{align*} 
As before, we can choose $K=K_1$ (large) so that the $H(\cdot)$  term in  every summand is bounded from below by $1/2$.  By the super-logarithmic growth of $k(n)$, we conclude that $\Pr[H_n(K,u)]\le n^{-2}$ provided $n$ is large, so that
the Borel-Cantelli lemma gives the result in this case too. 
\end{proof}

\begin{proof}[\it Proof of Proposition \ref{p:up}]
	If $\beta = \infty$ then let $u > \max_{j \in [\ell]}(a_j^{-1})$
	and set $r(n):= u (k(n)/n)^{1/d}$.
	If $\beta < \infty$ then let 
	$u> \max_{j \in [\ell]}( a_j^{-1} \hat H_\beta(b_j/d))$
	and set $r_n := (u(\log n)/n)^{1/d}$.
By Lemmas \ref{l:small} and \ref{l:big}, there exists $K \in \N$ such
	that almost surely, $E_n(K,u)\cup H_n(K,u)$ occurs for at most
	finitely many $n$. By Lemma \ref{l:sep}, if
	$M_{n,k} > r_n$ then $E_n(K,u) \cup H_n(K,u)$ occurs.
	Therefore 
	 $M_{n,k(n)}\le r_n$ for all large enough  $n$, almost surely,
	and the result follows. 
\end{proof}

\subsection{Proof of Theorem \ref{t:M}}

In this subsection we go back to the mathematical framework in
Section \ref{secdefs}; that is, we make the assumptions in
the statement of Theorem \ref{t:M}. In particular we return
to assuming $A$ is a convex polytope
in $\R^d$ with $d \geq 2$, and the probability measure
$\mu$ has a density $f$.
We shall check the conditions required in order to apply Proposition \ref{p:up}.

To check these conditions, we shall use the following  lemma  and notation. 

\begin{lemm}\cite[Lemma 6.12]{CovPaper}
	Suppose $\ph, \ph'$ are faces of $A$  with $D(\ph)>0$ and $D(\ph')=d-1$, and with $\ph\setminus \ph'\ne \emptyset$. Then $\ph^o\cap \ph'=\emptyset$ and $K(\ph,\ph')<\infty$, where we set
	\begin{align}
K(\ph,\ph'):= \sup_{x\in\ph^o} \frac{\dist(x,\partial \ph)}{\dist(x,\ph')}. 
		\label{e:Kpp}
	\end{align}
\end{lemm}
Now define
 \begin{align}
	 \label{e:Kdef}
	 K(A) := \max \{K(\ph, \ph'): \ph, \ph' \in \Phi(A), D(\ph) >0,
	 D(\ph') = d-1, \ph \setminus \ph' \neq \emptyset \}.
 \end{align}
Then $K(A)<\infty$ since $A$ is a finite polytope. 

For $j \in \{0,1,\ldots,d\}$
let  $\Phi_j(A)$ denote the collection of $j$-dimensional faces 
of $A$.  
For any $D\subset A$ and $r>0$ set $D_r = \{x\in A:  B(x,r)\cap D\neq\emptyset  \}$.

\begin{lemm}
\label{OTG}
	The restriction of Lebesgue measure to $A$ has the doubling property.
	Moreover  the
	conditions (O), (B), (T) and (G) 
	are met. 
\end{lemm}
\begin{proof}
	First we verify the doubling property.  
	 By the proof of Lemma \ref{l:volLB}, there exists $b >0$
	 such that $\inf_{x \in A,r \in (0,b]}r^{-d} \Vol(B(x,r) \cap A)
	 >0$. Since $\Vol(B(x,2r) \cap A) $ is at most
	 $ 2^d\theta_d r^d$ for $r \leq b$, and is at most $\Vol(A)$ for
	 all $r$, the doubling property follows.

%	 we have
%Observe first that taking $\mu_*= \lambda_d|_A $ gives us a doubling measure.
%Indeed, by Lemma \ref{l:poly_volume}-(iii) there exists $r_0 >0$ such that
%for all $x \in A$ and  $r \in (0,r_0)$ we have
%$\mu_*(B(x,r)) \geq (\min_{\ph \in \Phi(A)} \rho_\ph ) r^d$, while
%$\mu_*(B(x,2r)) \leq 2^d \theta_d r^d$.
%Moreover $\mu_*(B(x,r))$ and $\mu_*(B(x,2r))$ are bounded away from 0 and
%$\infty$, uniformly over $x \in A$ and $ r \geq r_0$.
%This establishes the doubling property.  Moreover, conditions
%	(O), (B), (T) and (G) are satisfied by Lemma \ref{OTG}.

Points of $A$ can be ordered by using the lexicographic ordering inherited from $\R^d$, thus (O). Since $A$ is convex, for all $x \in A$ and $r >0$
	the set $B(x,r) \cap A$ is convex and hence connected, implying (B).
	All convex polytopes are simply connected, and 
	therefore unicoherent \cite[Lemma 9.1]{RGG}, hence (T). 
Condition (G) follows immediately from Proposition \ref{p:geo},
which we prove below.
\end{proof}

\begin{prop}\label{p:geo}
Let $A$ be a convex finite polytope in $\R^d$.  Let $N(\cdot)$ denote the
number of components of a set.   There exists $\delta_1>0$ such that for any
$x \in A$ any $r \in (0,\delta_1)$, we have  $N(A\setminus B(x,r)) \le 2$.
Moreover,  in the case that $N(A\setminus B(x,r))=2$, the diameter of the
smaller component is at most $c r$, where $c$ is a constant depending
	only on $A$.  
\end{prop}

\begin{proof}[\it Proof of Proposition \ref{p:geo}]
	Write $B$ for $B(x,r)$.
	Our first observation
	is that if $y \in A \setminus B$, then  
	there is at least one vertex $v \in \Phi_0(A)$ such that the line
	segment $[y,v]$ is contained in $A \setminus B$. Indeed, if this failed
	then for each $v \in \Phi_0(A)$ there would exist a point 
	$u(v) \in [y,v] \cap B$. But then since $A$ is convex, $y $ would lie in
	the convex hull of $\{v:v \in \Phi_0(A) \}$, and therefore also in
	the convex hull of $\{u(v):v \in\Phi_0(A) \}$. 
	Indeed, there exist $\alpha_v\ge 0$ with $\sum_{v\in\Phi_0(A)} \alpha_v=1$ such that $y=\sum_{v\in\Phi_0(A)} \alpha_v v$, and there exists $\beta_v\in [0,1]$ such that $u(v)= \beta_v y + (1-\beta_v) v$. Substituting $v$ by $u(v)$ and rearranging terms shows that $y= \sum_{v} \alpha'_v u(v)$ with some nonnegative $\alpha'_v$ and $\sum_v \alpha'_v=1$, thus the claim.  
	But then since $B$ is convex we
	would have $y \in B$, a contradiction. 
	
	We refer to the one-dimensional faces $\ph \in \Phi_1(A)$
	as {\em edges} of $A$.  Our second observation is that if
	the number of edges of $A$ that intersect 
	$B$ is at most 1, then $A \setminus B$  
	is connected. Indeed, in this case, for any distinct
	$v,v' \in \Phi_0(A)$ there is a path along edges
	of $A$ from $v$ to $v'$ that avoids $B$. For example,
	if $v,v'$ lie in the same two-dimensional face $\ph$ of $A$
	then  since $B$ intersects at most one edge of the polygon
	$\ph$, there is a path from $v$ to $v'$ along the edges
	of $\ph$ avoiding $B$. Therefore all $v \in \Phi_0(A)$ lie
	in the same component of $A \setminus B$, so using the first
	observation  we deduce that $A \setminus B$ is connected.

	Recall the definition of $K(A)$ at (\ref{e:Kdef}).
	Our third observation is that if
	%	there is a finite constant $K^*$
	%	such that
	%for $r$ sufficiently small, 
	$\dist(v,B) \geq 3r K(A)$ for all $v \in \Phi_0(A)$
	then $A \setminus B$ is connected.
	Indeed, suppose $\dist (v,B) \geq 3rK(A)$ for all $v \in \Phi_0(A)$.
	Suppose
	$\ph,\ph'$ are distinct edges
	of $A$ with $B \cap \ph \neq \emptyset$, and pick
	$y \in B \cap \ph$. Then $\dist (y,\partial \ph) \geq  3r K(A)$
	so that by (\ref{e:Kpp}), $\dist(y,\ph') \geq 3 r K(A)/K(\ph,\ph') 
	\geq 3r$. Hence
	by the triangle inequality $\dist(B,\ph') \geq 3r -2r = r$,
	so that $B \cap \ph' = \emptyset$. Hence $B$ intersects at
	most one edge of $A$, and by our second observation $A \setminus B$
	is connected.
	
	Suppose $\dist(v,B) \leq 3rK(A)$ for some $v \in \Phi_0(A)$. 
	Provided $r$ is small enough, this cannot happen for more than
	one $v \in \Phi_0(A)$. If $u,u' \in \Phi_0(A) \setminus \{v\}$,
	then $v \notin [u,u']$ so $\dist(v,[u,u']) > 0$.
	Therefore provided $r$ is small enough, $[u,u'] \subset
	A \setminus B$.
	%then there is a path in $A \setminus B$ from $u$ to $u'$ by a
	%similar argument to our second observation. 
	%{\bf  [can we justify previous sentence?]}
	Thus provided $r$ is small enough, all vertices
	$u \in \Phi_0(A) \setminus \{v\}$ lie in the same component
	of $A \setminus B$. If also $v$ lies in this component, then
	(by our first observation) $A \setminus B$ is connected.
	
	Thus $A \setminus B$ is disconnected only if $v$ lies in
	a different component of $A \setminus B$ than all the other
	vertices. In that case, for $y \in A \setminus B$, if
	$[y,v] \subset A \setminus B$ then $y$ is in the same component
	as $v$; otherwise (by our first observation) $y$
	lies in the same component as all of the other vertices,
	and thus $A \setminus B$ has exactly two components.
	
	If $A \setminus B$ has two components, and $y \in A \setminus B$
	with $\|y-v\| > (3K(A)+2)r$, then we claim
	$[y,v] \cap B \neq \emptyset$.
	Indeed, for each $u \in \Phi_0(A) \setminus \{v\}$ the ray
	from $v $ in the direction of $u$ passes through $B$. But then
	by an argument based on the convexity of both $A$ and $B$,
	the ray from $v$ in the direction of $y$ must also pass through
	$B$. Since $\dist(v,B) \leq 3rK(A)$ and $\diam(B) = 2r$,
	this ray must pass through $B$ at a distance at most
	$(3K(A)+2)r$ from $v$, i.e. before it reaches $y$, and the claim
	follows. Therefore $y$ lies in the component of $A \setminus B$
	that does not contain $v$, and  thus the component containing $v$
	has diameter at most $(3K(A)+2)r$.
\end{proof}

To apply Proposition \ref{p:up}, we  need to define a partition of $A$ for each small $r>0$, then estimate the corresponding covering numbers and
$\mu$-measures in \eqref{0923}. 

Taking into account a variety of  boundary effects near $\partial A$,  one 
should consider separately regions near different faces of $A$. It is however
not trivial to construct this partition in such a way that we can obtain tight 
$\mu$-measure estimates in \eqref{0923}. 
The matter is complicated by the fact that the set $G$ in \eqref{0923} that intersects a region near $\varphi$ is potentially close to a lower dimensional face lying inside $\partial \varphi$.  
We can avoid the boundary complications by constructing inductively from regions near  to the highest dimensional face to the lowest, with increasing
'thickness'.  
The partition made of $T(\ph,r)$'s defined below and the left-over interior
region is defined for this purpose.

Let $(K_j)_{j\in \N}$ be an increasing sequence with $K_1=1$, and with
$K_{j+1}> (2K(A) +1) K_j$ for each $j \in \N$.  For instance,  we could
%take $K_{j}= j(2K+1)^{j-1}$.  
take $K_{j}= (2K(A)+2)^{j-1}$.

%Now for each $r>0$, $j \in [d]$ and $\ph \in \Phi_{d-j}(A)$,  define the set
Now for each $r>0$ and $\ph \in \Phi(A)$,  define the set
$$
%T(\ph,r) : = \ph_{r K_j } \setminus \cup_{\ph' \in \Phi(A)
T(\ph,r) : = \ph_{r K_{d- D(\ph)} } \setminus \cup_{\ph' \in \Phi(A)
%\setminus \{\ph\
: \ph' \subsetneq \ph} (\ph')_{r K_{d-D(\ph')}}, 
$$
where the $T$ stands for `territory'.
Also define 
$T(A,r):= A \setminus \cup_{\varphi \in \Phi(A)} \ph_{r K_{d-D(\ph)}}$
For each $\ph \in \Phi^*(A)$,
we have $T(\ph,r)\ne\emptyset$ for all $r$ sufficiently small.  Hence,  there exists $r_0>0$ such that for all $\ph$ and all $r<r_0$,   $T(\ph,r)\ne\emptyset$.   Moreover,
territories of distinct faces are disjoint, as we show in the following lemma.

\begin{lemm}\label{l:geo1} There exists $r_0 >0$ such that  for all $r \in
	(0,r_0)$,  and any distinct $\ph,\ph'\in\Phi^*(A)$,
	it holds that $T(\ph,r)\cap T(\ph',r)=\emptyset$.  
	Moreover, if $\ph, \ph'\in\Phi(A)$ with
	$\ph\setminus \ph'\neq \emptyset$, and $y\in T(\ph,r)$,
	then  $B(y,r)$ does not intersect $\ph'$.
\end{lemm}

\begin{proof}
	We can (and do) assume without loss of generality that
	$\ph\setminus \ph'\ne\emptyset$ and $\ph'\setminus \ph\ne\emptyset$.  
	Indeed, if $\ph \subset \ph'$,
	then by construction $T(\ph',r)\cap T(\ph,r)=\emptyset$.   

	%If both $\ph$ and $\ph'$ are vertices,  
	%then $T(\ph,r)\cap T(\ph',r)=\emptyset$ for all $r$ small.  
	%So it suffices to consider the case where $D(\ph')>0$. 
	If  $\ph$ is a vertex,  
	then 
	$\dist(\ph, \ph') >0$ so that
	$T(\ph,r)\cap T(\ph',r)=\emptyset$ for all $r$ small.  
	So it suffices to consider the case where 
	$D(\ph)>0$ and 
	$D(\ph')>0$.

%Let $j, j'\in [d]$ be such that $\ph\in \Phi_{d-j}(A)$ and 
%	$\ph'\in\Phi_{d-j'}(A)$.
	Let $j := d-D(\ph)$ and $j':= d-D(\ph')$.
	We can and do assume $j'\le j\le d-1$.
	%Let $\ph_1, ..., \ph_q \in \Phi_{d-1}(A)$ be such that $\ph=\cap_{i\in[q]} \ph_i$.  Then there exists some $i\in[q]$ such that $\ph'\setminus \ph_i\ne\emptyset$, for otherwise the condition $\ph'\setminus \ph\ne\emptyset$ is not satisfied.  
	
If there exists $x\in T(\ph,r)\cap T(\ph',r)$,  then we can
find $z\in \ph, z'\in \ph'$ such that  $\|x-z\|\le r K_j$ and 
$\|x-z'\|\le r K_{j'}$.  Therefore 
%$\dist(z',\ph_i)\le
	$\dist(z,\ph')\le r(K_j + K_{j'})\le 2rK_{j} $.
	
On the other hand, since $x \in T(\ph,r)$,
$\dist(x,\partial\ph)\geq rK_{j+1}$, and so by the triangle inequality,  
	$r K_{j+1}  - rK_{j} \le \dist(z,\partial \ph)\le K(A)\dist(z,\ph')$,
	where the last inequality comes from (\ref{e:Kpp}).
	Combining the estimates leads to  $K_{j+1}\le (2K(A)+1)K_{j}$,
	which is a contradiction.  The first claim follows. 

Moving to the second claim, let $\ph, \ph' \in \Phi(A)$ with 
$\ph \setminus \ph' \neq \emptyset$. Suppose $y \in \ph'_r$.  Set
$$
\tilde{\Phi} := \{ \psi \in \Phi(A): \psi \subsetneq \ph',
y \in \psi_{K_{D-d(\psi)}} \}.
$$
If $\tilde{\Phi} = \emptyset$ then $y \in T(\ph',r)$.
Otherwise, choose $\psi \in \tilde{\Phi}$ of minimal dimension.  Then 
	$y \in T(\psi,r)$. Either way, $y \notin T(\ph,r)$ by
	the first claim. Therefore $T(\ph,r) \cap \ph'_r = \emptyset$.
	%and the second claim follows.
%	
%Moving to the second claim,
%	let $y\in T(\ph,r)$ and $\ph'\in\Phi_{d-1}(A)$
%	with $\ph\setminus \ph'\neq \emptyset$.
%Then by (\ref{e:Kpp}), $\dist(y,\partial \ph) \leq K(\ph,\ph')\dist(y,\ph')$,
%and hence
%$$
%	\dist(y,\ph') \geq K^{-1} \dist(y ,\partial \ph)
%	\geq r K_{d-D(\ph)+1}/K > r,
%$$
%	so that  $B(y,r) \cap \ph' = \emptyset$.
%
%	we notice that  if the opposite holds, then there is $z'\in\ph'$ such that $\|x-z'\|\le r\le  K_{j'} r$ for any $j'$.   Hence repeating  the  argument gives the second claim as well.  
\end{proof}

As a last ingredient for applying Proposition \ref{p:up}, for each $J>1$ and $r\in(0,1)$, we construct a partition of $A$ and show  \eqref{0923}  for all $G$ with diameter at most $Jr$. The coefficients $a_j$ depend on the location of $G$ in relation to faces of $A$.   

 %We also here give a lower bound on the volume of balls intersected
 %with $A$, which will help to show our $\mu_*$ is a doubling 
 %measure. We write $\lambda_d$ for $d$-dimensional Lebesgue measure
 %(volume).

\begin{lemm}
	\label{l:poly_volume}
	Let $J\in \N$ and $\epsilon>0$. Then
	%For each small $r>0$, we have
	%the partition made of $T(\ph,2Jr), \ph\in\Phi(A)$ and
	%$\Theta_r = A\setminus (\cup_\ph T(\ph, 2Jr))$ satisfies
	the following hold: 
	%{\bf [MP: I don't know why there is a factor of $2J$ in the preceding
	%sentence but not in the conditions (i) and (ii) below]}
	\begin{itemize}
		\item[(i)]  
			For each $\ph \in \Phi(A)$ we have
			 $\kappa( T(\ph,2Jr), r)=O(r^{-D(\varphi)})$
			 as $r \downarrow 0$. Moreover we have
			%$\kappa( \Theta_r,r)=O(r^{-d})$ 
			$\kappa( A \setminus \cup_{\ph \in \Phi(A)}
			T(\ph,2Jr),r)=O(r^{-d})$ 
			as $r \downarrow 0$.
		\item[(ii)] For all small $r>0$ and
			any $G \subset A$ with $\diam(G)\le J r$, if it intersects $T(\ph,2Jr)$ for some $\varphi\in \Phi^*(A)$, then
		\begin{align}
			\label{0925a}
	\mu(G_r\setminus G)\ge (1-\epsilon) f_\varphi \rho_\varphi r^d.
		\end{align}
		%otherwise,  
		%\begin{align}
		%	\label{0925b}
		%	\mu(G_r\setminus G)\ge (1-\epsilon) f_0 \theta_d r^d. 
		%\end{align}
%\item[(iii)]
%For all small $r >0$ and  all $x \in A$,  we have
	%$\lambda_d(B(x,r) \cap A) \geq (\min_{\ph \in \Phi(A)} \rho_\ph) r^d$.
	\end{itemize}
\end{lemm}

\begin{proof}
Item (i) follows by the definition of $T(\ph, r)$.  Indeed, $\varphi$ is
contained in a bounded region within a $D(\varphi)$-dimensional affine space,
	and therefore can be covered by $O(r^{-D(\varphi)})$ balls of radius 
	$r$.  If we then take balls of radius 
	$r(1+ 2JK_{d-D(\varphi)})$ with
	the same centres, they will cover $T(\ph,2Jr)$,
	and one can then cover each of the larger balls with a fixed
	number of balls of radius $r$.
	
For (ii),  let $G \subset A$ with $\diam(G)\le Jr$.  Suppose first that
$G\cap T(\ph,2Jr)\ne\emptyset$ for some $\ph\in\Phi(A)$.
	Let $x_0\in G\cap T(\ph,2Jr)$.  Then 
$G_r \subset B(x_0,2Jr)$.  By Lemma \ref{l:geo1}, we see that $B(x_0, 2Jr)$ 
does not intersect any  $\ph'\in\Phi(A)$ with $\ph\setminus \ph'\ne\emptyset$. 
	  It follows that  
	  \begin{align}\label{1420}
	  	B(x_0, 2Jr)\cap A = B(x_0,2Jr)\cap (z_0 + \mathcal K_\ph )
	  \end{align}
	where $\mathcal K_\ph$ is the cone determined by $\ph$ and $z_0$ is the point of $\ph$ closest to $x_0$. 
	
	Set $D(x,r):= B(x,r)\cap (x + \mathcal K_\ph )$. 
	We claim that for any $x\in G$, we have $D(x,r)\subset A$. 
	Indeed, given  $y\in D(x,r)$, 
	we can write $y=z_0 + (x-z_0) + (y-x)=: z_0 + \theta_1 + \theta_2$.
	Here $\theta_1, \theta_2\in \mathcal K_\ph$.
	By convexity and scale invariance of $\mathcal K_\ph$,  we have
	$\theta_1+\theta_2 \in \mathcal K_\ph$
	so $y \in z_0+ \mathcal K_\ph$.
	Also $\|y-x_0\| \leq \|y-x\| + \|x-x_0\| \leq 2Jr$,
	and hence
	 $y\in A$ by \eqref{1420}, as claimed. 
	
	It follows that  (with $\oplus$ denoting Minkowski addition)
	\begin{align*}
		\mu(G_r \setminus G) \ge \mu((G\oplus D(o,r)) \setminus G)
		\ge  \Vol((G\oplus D(o,r)) \setminus G)
		\inf_{x \in G\oplus D(o,r)} f(x).
	\end{align*}
	By the Brunn-Minkowski inequality \cite[Section 5.3]{RGG}, we have
	$ \Vol(G\oplus D(o,r)) \ge \Vol(G) 
	+ \Vol(D(o,r)) = \Vol(G) + \rho_\ph r^d$.  The claim \eqref{0925a} follows by the continuity of $f$ on $\partial A$. 
	
As for the case $\varphi =A$, suppose now that $G \cap T(A,2Jr) \neq 
	\emptyset.$ Taking $x \in G \cap T(A,2Jr)$ we have
	$\dist(x,\partial A) \geq 2Jr$, and hence
	$\dist(G, \partial A) \geq 2Jr - Jr =Jr$. Therefore
	%If  $G$ does  not intersect any $T(\ph, 2Jr)$, then necessarily
	$G_r\subset A$, so
	 by the Brunn-Minkowski inequality 
	$$
	\mu(G_r \setminus G) \geq 
	%\mu ((G \oplus B(o,r)) \setminus G)
	 f_0 \Vol ((G \oplus B(o,r)) \setminus G)
	\geq f_0 \theta_d r^d.
	$$
	In this case $f_\varphi =f_0$ and $\rho_\ph = \theta_d$,
	and the claim \eqref{0925a}
	follows in this case too,
	 completing the proof of (ii). 
%
%For (iii) take $x \in A$.  If $x \in \Theta_r$ then $B(x,r) \subset A$ and
%$\lambda_d(B(x,r) \cap A) = \theta_d r^d$.  If $x \in T(\ph,2Jr)$ for some
%$\ph \in \Phi(A)$ then as above, provided $r$ is small enough 
%	$D(x,r) \subset  A$, so that $ \lambda_d(B(x,r) \cap A)
%	\geq \lambda_d(x + D(o,r)) = \rho_\ph r^d.$ This gives us (iii).
\end{proof}

\begin{proof}[\it Proof of Theorem \ref{t:M}]
	By (\ref{e:LleM}), and Theorem \ref{thmpolytope},
%As $L_{n,k}\le M_{n,k}$ for all $k,n$ with $n>k$, 
	it suffices to prove the upper bound. We shall do
	this by applying Proposition \ref{p:up} in the situation
	of Theorem \ref{t:M}.

%Observe first that taking $\mu_*= \lambda_d|_A $ gives us a doubling measure.
%Indeed, by Lemma \ref{l:poly_volume}-(iii) there exists $r_0 >0$ such that
%for all $x \in A$ and  $r \in (0,r_0)$ we have
%$\mu_*(B(x,r)) \geq (\min_{\ph \in \Phi(A)} \rho_\ph ) r^d$, while
%$\mu_*(B(x,2r)) \leq 2^d \theta_d r^d$.
%Moreover $\mu_*(B(x,r))$ and $\mu_*(B(x,2r))$ are bounded away from 0 and
%$\infty$, uniformly over $x \in A$ and $ r \geq r_0$.
	By Lemma \ref{OTG}, the restriction to $A$ of Lebesgue measure
	has the doubling property, and 
	conditions (O), (B), (T) and (G) are satisfied 

	To apply Proposition \ref{p:up}, we need to define
	(for each $K \in \N$ and each $r \in (0,r_0(K))$) a finite partition
	$\{T(j,K,r)\}$.
	For this we take the sets
	$T(\varphi,2Kr), \ph \in \Phi^*(A)$.
	%where we set $\Phi^*(A):= \Phi(A) \cup \{A\}$, and
	%we set
	%Note that $T(A,r):= A \setminus \cup_{\ph \in \Phi(A)} T(\ph,Kr)$.
	By Lemma \ref{l:geo1}, and the definition of $T(A,r)$,
	for each $K \in \N$
	there exists $r_0(K) >0$
	such that for $r \in (0,r_0(K))$
	the sets $T(\varphi,2Kr),$ $ \ph \in \Phi^*(A)$, do indeed 
	partition $A$.

	For each $\ph \in \Phi^*(A)$,
using Lemma \ref{l:poly_volume}-(i) we have 
	the condition  (\ref{1231a}) in Proposition \ref{p:up},
	where the constant denoted $b_j$ there is equal to
		$D(\varphi)$.
		%for each $\ph \in \Phi^*(A)$.
		 %and equals $d$ if $\ph = A$.
Also, using Lemma \ref{l:poly_volume}-(ii) we have the condition
	 (\ref{0923})
	in proposition \ref{p:up}, where the constant denoted $a_j$ there is 
	equal to $(1-\eps) f_\ph \rho_\ph$.
	%if $\ph \in\Phi(A)$,
	%and equals $(1-\eps) f_0 \theta_d$ if $\ph=A$.

	Suppose $\beta<\infty$. 
 By applying 
	%Lemma
	%\ref{OTG}, 
	%\ref{l:poly_volume} and
	Proposition \ref{p:up} in the manner described above
	%\ref{p:geo}, \ref{p:up},
	we see that for $\epsilon>0$, we have
\begin{align*}
\limsup_{n\to\infty} n (M_{n,k(n)})^d/\log n \le 
	%\max\Big(\frac{\hat H_\beta(1)}{(1-\epsilon)f_0 \theta_d},  
	\max_{\varphi\in\Phi^*(A)}
	\Big(
	\frac{\hat H_\beta(D(\varphi)/d)}{(1-\epsilon)f_\varphi \rho_\varphi} \Big),
\end{align*}
and the result follows.  If $\beta=\infty$, using corresponding part of Proposition \ref{p:up}  gives the result in this case too. 
\end{proof}

%\beg{acknowledgements}
%{\bf Acknowledgements.}
%I thank Alastair King and Andrew du Plessis
%for some helpful conversations during the 
%preparation of this paper. I thank  an anonymous referee,
%	and also Xiaochuan Yang, for reading through earlier 
%	versions and making some
%helpful observations.
%
%Data sharing not applicable to this article as no datasets were generated or analysed during the current study.
%%\end{acknowledgements}




\end{document}